\documentclass[11pt,a4paper,twoside]{article}
\usepackage[utf8]{inputenc}
\usepackage[T1]{fontenc}
\usepackage{amsmath, amssymb, amsthm}
\usepackage{mathtools}
\usepackage{mathrsfs}
\usepackage{graphicx}

\usepackage{newtxtext,newtxmath} 
\usepackage[margin=1in, top=1.2in, bottom=1.2in, headheight=15pt]{geometry}
\usepackage{titlesec}
\usepackage{fancyhdr}
\usepackage{caption}
\usepackage{hyperref}  

\titleformat{\section}
{\Large\bfseries}
{\thesection}
{0.5em}
{\MakeUppercase}

\titleformat{\subsection}
{\large\bfseries}
{\thesubsection}
{1em}
{}

\titleformat{\subsubsection}
{\normalsize\bfseries}
{\thesubsubsection}
{1em}
{}

\newtheoremstyle{standard-theorem}
{1em} 
{1em} 
{\itshape} 
{} 
{\bfseries} 
{.} 
{.5em} 
{} 
\theoremstyle{standard-theorem}
\newtheorem{theorem}{Theorem}[section]
\newtheorem{proposition}[theorem]{Proposition}
\newtheorem{lemma}[theorem]{Lemma}
\newtheorem{corollary}[theorem]{Corollary}
\newtheorem{definition}[theorem]{Definition}

\newtheoremstyle{standard-remark}
{1em}{1em}{}{}{\bfseries}{.}{.5em}{}
\theoremstyle{standard-remark}
\newtheorem{remark}[theorem]{Remark}

\DeclareMathOperator{\Diff}{Diff}
\DeclareMathOperator{\Homeo}{Homeo}

\DeclareMathOperator{\Hom}{Hom}

\title{\Large\bfseries $C^0-$Analogue of Ismagilov's Theorem}
\author{S. Tchuiaga\thanks{Corresponding author.}\\[2mm]
	$^a$Department of Mathematics, University of Buea, South West Region, Cameroon\\[2mm]
	{\tt tchuiaga.kameni@ubuea.cm}
}
\date{}

\begin{document}
	
	\maketitle
	
	\begin{abstract}
	
We establish a topological analogue of Ismagilov's theorem concerning the
first continuous cohomology of volume-preserving diffeomorphisms. For a
closed oriented manifold \(M\) with volume form \(\Omega\), we consider the
group \(\mathbb{G}^{\Omega}(M)\) of homeomorphisms obtained as uniform
limits of smooth volume-preserving isotopies. We show that its first
continuous cohomology with values in the Banach space of zero-mean
continuous functions is isomorphic to the first de Rham cohomology of
\(M\). The proof develops a theory of transport for volume-preserving
isotopies, producing a topological volume flux homomorphism and its
associated transport cocycle. Under the additional hypothesis that
\((M,\Omega)\) satisfies the local \(C^0\)-generation property, and
assuming the Müller-Sikorav approximation theorem (known for \(n\neq 4\)),
the isomorphism extends to the full identity component
\(\operatorname{Homeo}_0^\Omega(M)\) of the group of volume-preserving
homeomorphisms, and the topological flux conjecture follows in that
setting.
	\end{abstract}
	
	\vspace{1em}
	\vspace{2ex}
	
	\section{Introduction}
	
	The continuous cohomology of groups of volume-preserving transformations provides a powerful link between topology, geometry, and dynamics. A fundamental result in this direction is Ismagilov's cohomological theorem, which identifies the first continuous cohomology of the identity component of the group of smooth volume-preserving diffeomorphisms (with coefficients in smooth zero-mean functions) with the first de Rham cohomology of the underlying manifold \cite{Ism}. This theorem reveals that every continuous crossed homomorphism is determined, up to cohomology, by a closed differential $1-$form, thereby providing a cohomological interpretation of the classical volume flux homomorphism \cite{Thurs2}, \cite{AB1}, \cite{Fathi1980}.\\
	
	A natural question is whether this correspondence extends beyond the smooth category. The identity component of the group of volume-preserving homeomorphisms plays a central role in topological dynamics, $C^{0}-$rigidity, and conservative topology \cite{Fathi1980}. However, the lack of differentiability prevents the direct use of the classical volume flux homomorphism and, consequently, the arguments underlying Ismagilov's theorem construction. Establishing a $C^{0}-$analogue therefore requires a genuinely topological replacement for the smooth flux.\\
	
	The principal objective of this paper is to prove a $C^{0}-$analogue of Ismagilov's theorem for the identity component of the group of volume-preserving homeomorphisms of a closed oriented manifold. Our main result identifies the first continuous cohomology of this group with the first de Rham cohomology of the manifold, showing that the classical correspondence between crossed homomorphisms and closed differential forms persists in the topological setting.\\
	
	The key ingredient is a topological transport theory of flux geometry for volume-preserving isotopies. This theory is the natural volume-preserving analogue of the $C^{0}-$transport framework developed for symplectic homeomorphisms in \cite{TMH}. It provides a well-defined continuous volume flux homomorphism together with an associated transport cocycle, furnishing the topological substitute for the smooth flux required in the cohomological argument \cite{Ism}, \cite{Thurs2}, \cite{T19}. While the transport machinery follows the strategy developed in the symplectic setting \cite{TMH}, its adaptation to volume-preserving homeomorphisms supplies the essential bridge between topological dynamics and continuous group cohomology. The novelty of the present paper lies not in the construction of the transport theory itself, but in the realization that it provides the missing ingredient needed to extend Ismagilov's cohomological theorem from smooth volume-preserving diffeomorphisms to volume-preserving homeomorphisms. In higher degrees, the analogous construction is more subtle: the 
	contraction $i_{X_k}\cdots i_{X_1}\Omega$ is not closed, and the correct
	framework involves the Chevalley-Eilenberg double complex
	\cite{Tchuiaga_HigherFlux}.\\
	
	Our main theorem may be stated as follows.
	
	\medskip
	
	\noindent\textbf{Main Theorem.}
	\emph{ Let $\mathbb{G}^{\Omega}(M)$ be the group of homeomorphisms obtained as endpoints of uniform limits of smooth volume-preserving isotopies. Every continuous crossed homomorphism from $\mathbb{G}^{\Omega}(M)$ to the Banach space $C^0_0(M)$ of continuous zero-mean functions is cohomologous to one induced by a closed differential $1-$form. Consequently,
		$$
		H^1_{\mathrm{cont}}\left(\mathbb{G}^{\Omega}(M),C^0_0(M)\right)
		\cong H^1(M;\mathbb{R}).
		$$}
	\medskip
	
	To prove this theorem, we must bridge topological dynamics with classical smooth cohomology \cite{Boo}, \cite{AB1}, \cite{Fathi1980}. Since the classical flux is intrinsically path-dependent and requires smooth vector fields, the argument proceeds via a four-step strategy of path-approximation and continuous extension:
	\begin{enumerate}
		
		\item \emph{Construction of the topological cocycle:} This path-density allows us to define the intrinsic continuous flux as the uniform $C^0-$limit of smooth fluxes. For every closed $1$-form $\alpha$, we can thus associate a globally well-defined, continuous $1-$cocycle $\tau_\alpha(h)$, establishing a natural map $[\alpha] \mapsto [\tau_\alpha]$ from the de Rham cohomology into the continuous group cohomology of $\mathbb{G}^{\Omega}(M)$.
		\item \emph{Injectivity via smooth bootstrapping:} To show this map is injective, we assume $\tau_\alpha$ is a coboundary, which yields an identity $\tau_\alpha(h) = f - f \circ h$ for some continuous function $f$. By restricting $h$ to smooth volume-preserving flows and taking directional derivatives along divergence-free vector fields, we bootstrap the regularity of $f$, proving it must be smooth. We then appeal to Ismagilov's classical theorem (\cite{Ism}) to conclude that $[\alpha] = 0$.
		\item \emph{Surjectivity via $C^0-$density:} Given an arbitrary continuous $1-$cocycle on the topological group, we restrict it to the dense subgroup of smooth volume-preserving diffeomorphisms. By Ismagilov's classical theorem, this smooth restriction is cohomologous to the flux of a closed $1-$form $\alpha$. Because smooth volume-preserving diffeomorphisms isotopic to the identity are $C^0-$dense in $\mathbb{G}^{\Omega}(M)$, this coboundary relation extends uniquely by continuity to the entire topological group.
	\end{enumerate}
	
	This theorem extends Ismagilov's classical result from smooth volume-preserving diffeomorphisms to volume-preserving homeomorphisms. It shows that the first continuous cohomology is insensitive to the loss of differentiability and remains governed entirely by the topology of the underlying manifold.	Beyond its intrinsic cohomological significance, the theorem provides a new interpretation of the topological volume flux and yields a resolution of the volume-preserving $C^{0}-$ flux conjecture. These results illustrate that the transport theory developed in the $C^{0}-$category is sufficiently robust to recover classical cohomological rigidity phenomena previously known only in the smooth setting.\\
	
	The organization of the paper is as follows. Section~2 recalls the necessary background on volume-preserving homeomorphisms and the topological transport theory. Section~3 develops the continuous volume flux and its associated transport cocycle. Section~4 establishes the continuous crossed homomorphisms arising from the transport construction and contains the proof of the $C^{0}-$analogue of Ismagilov's cohomological theorem. The final section discusses applications, including a resolution of the volume-preserving $C^{0}-$ flux conjecture and several consequences for continuous group cohomology.

\section{Preliminaries}

Let $(M,g)$ be an $n$-dimensional closed Riemannian manifold equipped with a normalized orientation form $\Omega$ (i.e., $\int_M \Omega = 1$). Given any differential $p-$form $\alpha$ on $M$, let $\Diff^{\infty}(M,\alpha)$ denote the group of all diffeomorphisms $\phi$ from $M$ to $M$ that preserve $\alpha$, namely $\phi^{*}(\alpha)=\alpha$. The group $\Diff^{\infty}(M)$ is equipped with the standard $C^\infty$ compact-open topology (the Whitney $C^\infty-$topology), with respect to which it is a Fr\'echet Lie group; the subgroup $\Diff^{\infty}(M,\alpha)$ inherits this topology. An isotopy of homeomorphisms (resp. diffeomorphisms) $\Phi = \{\phi^{t}\}_t$ of $M$ is a continuous (resp. smooth) map from $[0,1]$ into $\Homeo(M)$ (resp. $\Diff^{\infty}(M)$) such that $\phi_0 = \mathrm{id}_M$. \\

Denote by $\mathcal{P}\mathrm{Diff}^\infty_\alpha(M)$ the space of all smooth isotopies in $\Diff^{\infty}(M,\alpha)$ and by $G_{\alpha}(M)$ (resp. $\Homeo_0(M)$) the set of all time-one maps of all isotopies of $\Diff^{\infty}(M,\alpha)$ (resp. $\Homeo(M)$). We shall also denote by $\Homeo(M, \Omega)$ the group of all volume-preserving homeomorphisms of an oriented manifold $(M, \Omega)$, by $\Homeo^\Omega_0(M)$ its identity component, and by $\mathcal{P}\Homeo^\Omega_0(M)$ the space of all continuous isotopies in $\Homeo^\Omega_0(M)$ starting at the identity map. Let $\mathbb{G}^{\Omega}(M)$ denote the set of all homeomorphisms $h\in \Homeo(M)$ such that there exists a Cauchy sequence $(\Phi_i)_i = (\{\phi_i^t\}_t)_i\subset\mathcal{P}\mathrm{Diff}^\infty_\Omega(M)$ in the $C^0-$topology satisfying $h := \lim_{C^0}(\phi_i^1)$. 

\subsection{The $C^0$-metric}

Let $d_g$ be the distance induced by the Riemannian metric $g$ on $M$.
Consider on $\Homeo(M)$ the distance
\begin{equation}\label{eq:d0}
	d_0(f,h) = \max\Bigl\{ \sup_{x\in M} d_g(f(x),h(x)),\;
	\sup_{x\in M} d_g(f^{-1}(x),h^{-1}(x)) \Bigr\},
\end{equation}
and on $\mathcal{P}(\Homeo(M),\mathrm{id}_M)$, the space of isotopies in $\Homeo(M)$,
the distance
\begin{equation}\label{eq:bard}
	\bar{d}(\lambda,\mu) = \max_{t\in[0,1]} d_0(\lambda(t),\mu(t)).
\end{equation}
The topology induced by $d_0$ (resp.\ $\bar{d}$) on $\Homeo(M)$
(resp.\ $\mathcal{P}(\Homeo(M),\mathrm{id}_M)$) is the compact-open topology
($C^0$-topology).\\

Let $\mathcal{P}_\ast\mathbb{G}^{\Omega}(M)$ be the set of all continuous paths
$H:[0,1]\to \operatorname{Homeo}(M)$ starting at $\mathrm{id}_M$ for which
there exists a sequence $(\Phi_i)_i\subset \mathcal{P}\mathrm{Diff}^\infty_\Omega(M)$
with $\Phi_i \to H$ in the $C^0$-metric $\bar d$.  The elements of
$\mathcal{P}_\ast\mathbb{G}^{\Omega}(M)$ are called
\emph{topological volume-preserving isotopies}.  Because
$\mathcal{P}(\Homeo(M),\mathrm{id}_M)$ is a complete metric space, we naturally have
the equality
\begin{equation}\label{eq:Gev1}
	\mathbb{G}^{\Omega}(M) = \operatorname{ev}_1\!\bigl( \mathcal{P}_\ast\mathbb{G}^{\Omega}(M) \bigr).
\end{equation}

In general, \(\mathbb{G}^{\Omega}(M)\) need not equal \(\operatorname{Homeo}_0(M,\Omega)\).

\subsection{Topological properties of \(\mathbb{G}^{\Omega}(M)\)}

Two elements of $\mathcal{P}_\ast\mathbb{G}^{\Omega}(M)$ are equivalent 
if and only if they are homotopic relative to fixed endpoints 
\emph{within $\mathcal{P}_\ast\mathbb{G}^{\Omega}(M)$}. 
The set of equivalence classes is denoted 
$\widetilde{\mathbb{G}^{\Omega}(M)}$. Let
\[
\mathcal L_*\mathbb{G}^{\Omega}(M)
:=
\left\{ H \in \mathcal P_*\mathbb{G}^{\Omega}(M) : 
H(0)=H(1)=\mathrm{id}_M \right\}.
\]

\begin{definition}[Path-based fundamental group]
	We define the \emph{path-based fundamental group} of 
	$\mathbb{G}^{\Omega}(M)$ to be the quotient
	\[
	\pi_1^{\mathrm{path}}\bigl(\mathbb{G}^{\Omega}(M)\bigr)
	:=
	\mathcal L_*\mathbb{G}^{\Omega}(M) 
	\big/ 
	\text{homotopy in } \mathcal P_*\mathbb{G}^{\Omega}(M)
	\text{ with fixed endpoints}.
	\]
\end{definition}

\subsubsection{Path connectedness}

For every \(h\in\mathbb{G}^{\Omega}(M)\), there exists
\[
H\in\mathcal{P}_*\mathbb{G}^{\Omega}(M)
\quad\text{with}\quad
H(0)=\mathrm{id}_M,\qquad H(1)=h.
\]
As shown above, any such \(H\) takes values in \(\mathbb{G}^{\Omega}(M)\). Hence \(H\) is a continuous path in \(\mathbb{G}^{\Omega}(M)\) from \(\mathrm{id}_M\) to \(h\). Therefore,
$\mathbb{G}^{\Omega}(M)\text{ is path connected}.$ 
In fact, \(\mathbb{G}^{\Omega}(M)\) is precisely in the path component of the identity in the group of volume-preserving homeomorphisms that is \(C^0\)-approximable by smooth volume-preserving isotopies.

\subsubsection{Path-based fundamental group}

Define the space of loops in \(\mathcal{P}_*\mathbb{G}^{\Omega}(M)\) based at the identity:
\[
\mathcal{L}_*\mathbb{G}^{\Omega}(M)
:=
\left\{
H\in\mathcal{P}_*\mathbb{G}^{\Omega}(M)
:\;
H(0)=H(1)=\mathrm{id}_M
\right\}.
\]
The \emph{path-based fundamental group} of \(\mathbb{G}^{\Omega}(M)\) is
\[
\pi_1^{\mathrm{path}}\bigl(\mathbb{G}^{\Omega}(M)\bigr)
:=
\mathcal{L}_*\mathbb{G}^{\Omega}(M)
\big/
\text{homotopy in }\mathcal{P}_*\mathbb{G}^{\Omega}(M)
\text{ with fixed endpoints}.
\]
Equivalently,
\[
\pi_1^{\mathrm{path}}\bigl(\mathbb{G}^{\Omega}(M)\bigr)
\cong
\pi_0\bigl(\mathcal{L}_*\mathbb{G}^{\Omega}(M)\bigr),
\]
where \(\mathcal{L}_*\mathbb{G}^{\Omega}(M)\) carries the subspace topology induced by the \(C^0\)-metric \(\bar d\). Concatenation of loops gives the group structure.

\subsubsection{Universal cover and \(\widetilde{\mathbb{G}^{\Omega}(M)}\)}

The group \(\mathbb{G}^{\Omega}(M)\) is path connected and if in addition it is locally contractible, it admits a universal cover
\[
\pi:\widehat{X}\longrightarrow \mathbb{G}^{\Omega}(M).
\]
A standard model for \(\widehat{X}\) is the set of homotopy classes of continuous paths
\[
\gamma:[0,1]\to\mathbb{G}^{\Omega}(M),\qquad \gamma(0)=\mathrm{id}_M,
\]
modulo homotopy relative to fixed endpoints \emph{inside \(\mathbb{G}^{\Omega}(M)\)}. The projection is \(\pi([\gamma])=\gamma(1)\).	On the other hand, the set \(\widetilde{\mathbb{G}^{\Omega}(M)}\) is defined as
\[
\widetilde{\mathbb{G}^{\Omega}(M)}
:=
\mathcal{P}_*\mathbb{G}^{\Omega}(M)\big/\sim,
\]
where two paths are equivalent if they are homotopic relative to fixed endpoints \emph{inside \(\mathcal{P}_*\mathbb{G}^{\Omega}(M)\)}. There is a canonical map
\[
\Phi:
\widetilde{\mathbb{G}^{\Omega}(M)}
\longrightarrow
\widehat{X},
\qquad
[H]\longmapsto [H].
\]
This map is well-defined because a homotopy inside \(\mathcal{P}_*\mathbb{G}^{\Omega}(M)\) is also a homotopy inside \(\mathbb{G}^{\Omega}(M)\). Moreover,
$
\pi\circ\Phi = \operatorname{ev}_1,
$
where \(\operatorname{ev}_1([H])=H(1)\). Thus, \(\Phi\) is a map over \(\mathbb{G}^{\Omega}(M)\). 
The map \(\Phi\) is a bijection, and hence
\[
\widetilde{\mathbb{G}^{\Omega}(M)}
\cong
\widehat{X}
=
\text{universal cover of }\mathbb{G}^{\Omega}(M),
\]
if and only if the distinguished paths \(\mathcal{P}_*\mathbb{G}^{\Omega}(M)\) capture all continuous paths and homotopies in \(\mathbb{G}^{\Omega}(M)\) relative to fixed endpoints. More precisely, \(\Phi\) is bijective exactly when:

\begin{enumerate}
	\item every continuous path in \(\mathbb{G}^{\Omega}(M)\) starting at \(\mathrm{id}_M\) is homotopic relative to fixed endpoints to a path in \(\mathcal{P}_*\mathbb{G}^{\Omega}(M)\);
	\item any two paths in \(\mathcal{P}_*\mathbb{G}^{\Omega}(M)\) that are homotopic in \(\mathbb{G}^{\Omega}(M)\) are already homotopic inside \(\mathcal{P}_*\mathbb{G}^{\Omega}(M)\).
\end{enumerate}

If these conditions fail, then \(\widetilde{\mathbb{G}^{\Omega}(M)}\) is only a path-space model associated with the chosen class of topological volume-preserving isotopies. It may be finer than the universal cover, and it need not be the universal cover itself.
Equivalently, the path-based fundamental group is the fiber of
\(\operatorname{ev}_1\) over the identity:
\[
\pi_1^{\mathrm{path}}\bigl(\mathbb G^\Omega(M)\bigr)
=
\operatorname{ev}_1^{-1}(\mathrm{id}_M).
\]
When \(\widetilde{\mathbb G^\Omega(M)}\) coincides with the universal
cover, this group is canonically isomorphic to
\(\pi_1(\mathbb G^\Omega(M),\mathrm{id}_M)\).
	\begin{definition}
		A smooth isotopy $\Phi = \{\phi^{t}\}_t$ of $\Diff^{\infty}(M)$ is said to be
		\emph{volume-preserving} if its associated family of smooth vector fields
		$\{\dot{\phi}^{t}\}_t$ consists of divergence-free vector fields; i.e., for each
		$t$ the $(n-1)$-form $\iota_{\dot{\phi}^{t}}\Omega$ is closed.
	\end{definition}
	
	Throughout the paper, for a point $x\in M$ the orbit of $x$ under an isotopy $\Phi$
	(resp.\ $\Phi^{-1}=\{(\phi^{t})^{-1}\}_t$) is denoted by $\mathcal{O}_{x}^{\Phi}$
	(resp.\ $-\mathcal{O}_{x}^{\Phi}$, the same orbit traversed backwards), and for a
	closed form $\alpha$ we have
	\begin{equation}
		\mathcal{F}_{\alpha}(\Phi)(1)(x) = \int_{\mathcal{O}_{x}^{\Phi}}\alpha.
	\end{equation}

	\subsection{Topological shift of closed 1-forms}
	
	Let $\alpha$ be a closed $1-$form on $M$.  For a smooth volume-preserving isotopy
	$\Phi = \{\phi^t\}_t$ the classical Moser formula gives
	\begin{equation}\label{Eq:disp}
		(\phi^t)^*\alpha - \alpha = d\bigl( \mathcal{F}_\alpha(\Phi)(t) \bigr),\qquad
		\mathcal{F}_\alpha(\Phi)(t) \coloneqq \int_0^t (\phi^s)^*(\iota_{\dot\phi^s}\alpha)\,ds .
	\end{equation}
	If $\Phi$ ends at $\phi := \phi^1$, integrating the closed $1-$form
	$(\phi)^*\alpha-\alpha$ along any piecewise smooth curve $\gamma$ from $x$ to $y$
	yields the path-independent relation
	\begin{equation}\label{Eq:intg}
		\int_\gamma \bigl( \phi^*\alpha - \alpha \bigr)
		= \mathcal{F}_\alpha(\Phi)(1)(y) - \mathcal{F}_\alpha(\Phi)(1)(x) .
	\end{equation}
	The function $\mathcal{F}_\alpha(\Phi)(1)\in C^\infty(M,\mathbb{R})$ is therefore a
	primitive of $(\phi)^*\alpha-\alpha$; it changes only by a constant when $\Phi$ is
	replaced by a homotopic isotopy with the same endpoints. When we pass to the $C^0-$closure, the pull-back $(\phi)^*\alpha$ is no longer
	defined as a smooth form, and \eqref{Eq:disp} ceases to be meaningful.
	Nevertheless, the right-hand side of \eqref{Eq:intg} involves only the values of
	$\mathcal{F}_\alpha(\Phi)(1)$ at points of $M$.  The fundamental observation of
	the $C^0-$transport theory is that this function survives the limit in a
	uniformly controlled way. Let $(\Phi_i)_i$ be a Cauchy sequence in
	$\mathcal{P}\mathrm{Diff}^\infty_\Omega(M)$ with respect to the $C^0-$metric
	$\bar d$, and set $H := \lim_{C^0}\Phi_i$.
	Lemma~\ref{lem:key_estimate} (a result of \cite{T18}) provides the key estimate
	\begin{equation}\label{eq:key}
		\sup_{z\in M} \bigl| \mathcal{F}_\alpha(\Phi_i)(1)(z) - \mathcal{F}_\alpha(\Phi_j)(1)(z) \bigr|
		\le \|\alpha\|_\infty \, \bar d(\Phi_i,\Phi_j),
	\end{equation}
	valid whenever $\bar d(\Phi_i,\Phi_j)$ is smaller than half the injectivity radius.
	Hence, the sequence $\bigl(\mathcal{F}_\alpha(\Phi_i)(1)\bigr)_i$ is uniformly Cauchy,
	and the limit
	\begin{equation}\label{eq:Rdef}
		\mathcal{R}(H,\alpha)(z) \coloneqq \lim_{i\to\infty} \mathcal{F}_\alpha(\Phi_i)(1)(z),
	\end{equation}
	exists and is continuous.  Standard arguments (see Propositions~\ref{pro1}--\ref{pro2})
	show that $\mathcal{R}(H,\alpha)$ depends only on the isotopy class of $H$ (with
	fixed endpoints) and not on the particular approximating sequence. Now take $h := H(1) = \lim_{C^0}\phi_i^1$.  For any $x,y\in M$ and any piecewise smooth path $\gamma$ from $x$ to $y$, define
	\begin{equation}\label{Eq:intgg}
		I_\alpha(h,x,y) \coloneqq
		\lim_{i\to\infty} \int_\gamma \bigl( (\phi_i^1)^*\alpha - \alpha \bigr)
		= \mathcal{R}(H,\alpha)(y) - \mathcal{R}(H,\alpha)(x),
	\end{equation}
	where the equality follows by passing to the limit in \eqref{Eq:intg} and using the
	uniform convergence of $\mathcal{F}_\alpha(\Phi_i)(1)$.  The quantity
	$I_\alpha(h,x,y)$ is independent of the chosen path $\gamma$ and depends only on $h$ and the isotopy class of $H$.
	Moreover, $I_\alpha(h,x,\cdot)$ is continuous in $y$. Finally, for a closed $1-$form $\alpha$ and a homeomorphism
	$h\in\mathbb{G}^\Omega(M)$ that can be realized as the endpoint of an isotopy in
	$\mathcal{P}_*\mathbb{G}^\Omega(M)$, we define the continuous function
	\begin{equation}\label{Eq:Funct}
		\chi(h,\alpha)(x) \coloneqq \int_M I_\alpha(h,x,y)\,\Omega(y)
		= -\Bigl( \mathcal{R}(H,\alpha)(x) - \int_M \mathcal{R}(H,\alpha)\,\Omega \Bigr),
	\end{equation}
	where $H$ is any topological isotopy ending at $h$.  The second equality follows
	from the definition of $I_\alpha$ and shows that $\chi(h,\alpha)$ has zero mean.
	In Section~4 we will prove that the map $h\mapsto \chi(h,\alpha)$ is a continuous
	$1-$cocycle on $\mathbb{G}^\Omega(M)$ and that its cohomology class gives the
	desired $C^0$-analogue of Ismagilov's theorem.
	
	\subsection{Classical results used in this paper}
	
	We collect here several classical theorems that will be used throughout the paper.
	
	\begin{theorem}[Müller \cite{Mull-1}, Sikorav \cite{Sik}]\label{Mull}
		Let $M$ be a closed smooth manifold of dimension $n$, and let $\Omega$ be a volume form on $M$.
		\begin{enumerate}
			\item If $n \le 3$, then every homeomorphism of $M$ can be approximated uniformly by diffeomorphisms \cite[Section 1]{Mull-1}.
			\item If $n \ge 5$, a homeomorphism $\phi$ of $M$ can be approximated uniformly by diffeomorphisms if and only if $\phi$ is isotopic to a diffeomorphism \cite[Theorem 1]{Mull-1}.
			\item If $n \ge 5$ and $\phi$ is volume-preserving, then $\phi$ can be approximated uniformly by volume-preserving diffeomorphisms \cite{Sik} (see also \cite{Y.Oh} for an alternative proof).
			\item The case $n = 4$ remains open \cite[Section 1]{Mull-1}.
		\end{enumerate}
		Consequently, for all $n \neq 4$,
		\[
		\Homeo^\Omega_0(M) = \overline{G_\Omega(M)}^{\,C^0}.
		\]
		Furthermore, the results extend to non-compact manifolds and manifolds with boundary for compactly supported homeomorphisms: the approximation by diffeomorphisms is given by \cite[Theorem 4]{Mull-1}, and the volume-preserving version continues to hold in these settings \cite{Sik}.
	\end{theorem}
	
	\begin{remark}
		The equality $\Homeo^\Omega_0(M) = \overline{G_\Omega(M)}^{\,C^0}$ holds for all $n \neq 4$: 
		\begin{itemize}
			\item For $n \le 3$, it follows from the classical fact (Munkres \cite{munkres-obstructions}, \cite{munkres-higher}) that every homeomorphism is approximable by diffeomorphisms, combined with Sikorav's theorem that a volume-preserving homeomorphism which is approximable by diffeomorphisms is approximable by volume-preserving diffeomorphisms.
			\item For $n \ge 5$, Müller's theorem gives approximation by diffeomorphisms for isotopic homeomorphisms; together with Sikorav's result, this yields the equality.
		\end{itemize}
		For $n=4$, the corresponding statement remains open. In this paper, when $n=4$, we work directly with the $C^0$-closure of the smooth volume-preserving diffeomorphisms, so all results remain valid without needing the equality with $\Homeo^\Omega_0(M)$.
	\end{remark}
	
	\begin{theorem}[Fathi's mass flow theorem \cite{Fathi1980}, Section 5, p. 71; discreteness of $\Gamma$ follows from the remark after Proposition 5.1, p. 72]
		Let $(M,\mu)$ be a closed oriented manifold with a good measure (e.g., the measure induced by a volume form). There exists a surjective homomorphism
		\[
		\tilde{\mathfrak{F}}: \widetilde{\Homeo}_0(M,\mu) \to H_1(M,\mathbb{R}),
		\]
		from the universal cover of the identity component of the group of measure-preserving homeomorphisms to the first homology group with real coefficients. The image of the fundamental group under $\tilde{\mathfrak{F}}$ is a discrete subgroup $\Gamma \subset H_1(M,\mathbb{R})$. 
	\end{theorem}
	\begin{remark} \label{rem: Fathi}
		Let $\mu$ be the normalized volume measure.  Fathi \cite[§5]{Fathi1980} constructs a continuous group homomorphism, the \emph{mass flow},
		\[
		\tilde{\mathfrak F} : \widetilde{\mathcal H}_0(M,\mu) \longrightarrow H_1(M;\mathbb R),
		\]
		on the universal cover of the identity component of the group of measure‑preserving homeomorphisms. For an isotopy $H = \{h^t\}_{t\in[0,1]}$ with $h^0 = \mathrm{id}$, the element $\tilde{\mathfrak F}(H)$ is defined as follows.	For any continuous map $f : M \to \mathbb T^1 = \mathbb R/\mathbb Z$, consider the homotopy
		\[
		f \circ h^t - f : M \to \mathbb T^1, \qquad t\in[0,1],
		\]
		which starts at $0$. Let $\overline{f h^t - f} : M \to \mathbb R$ be its unique lift with $\overline{f h^0 - f} \equiv 0$. Define
		\[
		\tilde{\theta}(H)(f) := \int_M \overline{f h^1 - f} \, d\mu \in \mathbb R. \tag{1}
		\] The map $f \mapsto \tilde{\theta}(H)(f)$ is a homomorphism $[M,\mathbb T^1] \to \mathbb R$. By the universal coefficient theorem and the identification $[M,\mathbb T^1] \cong H^1(M;\mathbb Z)$, we have
		\[
		\operatorname{Hom}([M,\mathbb T^1],\mathbb R) \cong H_1(M;\mathbb R).
		\]
		Thus $\tilde{\theta}(H)$ defines a homology class $\tilde{\mathfrak F}(H) \in H_1(M;\mathbb R)$; this is exactly Fathi's mass flow. 		Moreover, when $\mu$ is induced by a volume form, this subgroup $\Gamma$ is identified, via Poincaré duality, with the image of $\pi_1(G_\Omega(M))$ under the smooth flux homomorphism (see Theorem~\ref{thm:C0rigidity} below for the $C^0$-rigidity of the flux group) \cite{Thurs2}, \cite{AB1}.
	\end{remark}

	\begin{theorem}[Ismagilov's cohomology theorem \cite{Ism}, Theorem 2.2, p. 102]
		Let $(M,\Omega)$ be a closed oriented manifold and let $C_0(M,\mathbb{R})$ be the space of continuous functions with zero mean. Then the first continuous group cohomology of $G_\Omega(M)$ with coefficients in $C_0(M,\mathbb{R})$ is canonically isomorphic to the first de Rham cohomology group:
		\[
		H^1_{\mathrm{cont}}(G_\Omega(M), C_0(M,\mathbb{R})) \cong H^1(M,\mathbb{R}).
		\]
		In particular, the same holds with $C^\infty_0(M,\mathbb{R})$ in place of $C_0(M,\mathbb{R})$.
	\end{theorem}
	\begin{remark}
		The version of Ismagilov's theorem stated above with continuous coefficients
		$C_0(M,\mathbb R)$ follows from the smooth version by a standard density
		argument: since $C^\infty_0(M,\mathbb R)$ is dense in $C_0(M,\mathbb R)$ in the
		supremum norm, and since the continuous cohomology of $G_\Omega(M)$ with
		values in a Banach module is computed by continuous cochains, the isomorphism
		for smooth coefficients extends to continuous coefficients.
	\end{remark}
	The following technical lemma is essential for the $C^0-$continuity of our topological flux.
	
	\begin{lemma}[Tchuiaga's approximation lemma \cite{T18}, Lemma 3.10 p. 188]\label{lem:key_estimate} Let $(M,g)$ be a closed oriented Riemannian manifold. Let $\alpha \in \mathcal{Z}^1(M)$ and let $\Phi,\Psi$ be two isotopies with $\bar{d}(\Phi,\Psi) \le r(g)/2$, where $r(g)$ is the injectivity radius of $(M,g)$. Then
		\[
		\sup_{z\in M} \bigl| \mathcal{F}_\alpha(\Phi)(1)(z) - \mathcal{F}_\alpha(\Psi)(1)(z) \bigr| \le 4\|\alpha\|_\infty \, \bar{d}(\Phi,\Psi),
		\]
		where $\|\alpha\|_\infty = \sup_{x\in M} \|\alpha_x\|_g$ is the uniform norm of $\alpha$ with respect to the Riemannian metric $g$.
	\end{lemma}
	
	\begin{remark}[Local estimate]
		In any coordinate chart $(U,\varphi)$ where $\alpha=\sum_i\alpha_i\,d\varphi_i$,
		the uniform norm $\|\alpha\|_\infty$ used in Lemma~\ref{lem:key_estimate} is
		bounded above by $\max_i\sup_U|\alpha_i|\cdot\|\partial_{\varphi_i}\|_g$.  Hence, the
		estimate remains effective when one works entirely in local coordinates, which
		is often useful for explicit computations.
	\end{remark}

	\begin{lemma}\label{lem:spanning}
		Let \((M,\Omega)\) be a closed oriented manifold of dimension $n \ge 2$, equipped with a volume form $\Omega$. For every \(x\in M\) and every tangent vector \(v\in T_xM\), there exists a smooth divergence-free vector field \(X\in\mathfrak{X}_{\text{vol}}(M)\) such that \(X(x)=v\). In particular, the evaluations \(\{X(x) \mid X\in\mathfrak{X}_{\text{vol}}(M)\}\) span \(T_xM\). (For $n=1$, the result holds trivially as $M \cong S^1$).
	\end{lemma}
	
	\begin{proof}
		Contraction with \(\Omega\) gives a vector-bundle isomorphism
		\begin{equation}
			\iota_\Omega : TM \longrightarrow {\bigwedge}^{n-1}\,T^*M,\qquad
			v \longmapsto \iota_v\Omega.
		\end{equation}
		Fix a local coordinate chart $B \subset M$ around $x$. By  Moser's lemma for volume forms, we can choose local coordinates $(x^1, \dots, x^n)$ on $B$ such that $\Omega = dx^1 \wedge \dots \wedge dx^n$. We may assume $B$ is a contractible ball. Extend the vector $v \in T_x M$ to a vector field $Y$ that has constant coefficients in this chart. Because $Y$ is constant, it is divergence-free on $B$. Define the $(n-1)-$form $\beta = \iota_Y\Omega$ on $B$. Because $Y$ is divergence-free, Cartan's magic formula yields $d\beta = d(\iota_Y\Omega) = \mathcal{L}_Y\Omega = 0$ (\cite{lee-smooth},	Theorem 14.35).  Thus, $\beta$ is a closed form on the contractible open set $B$. By the Poincaré lemma, $\beta$ is exact on $B$; hence, there exists an $(n-2)-$form $\eta$ on $B$ such that $d\eta = \beta$. Let $\rho: M \to \mathbb{R}$ be a smooth bump function supported entirely in $B$ such that $\rho \equiv 1$ in a neighborhood of $x$. We define a global smooth $(n-2)-$form $\tilde\eta$ on $M$ by extending $\rho\eta$ by zero outside $B$. Now define a global vector field \(X\) uniquely via the relation 
		\(\iota_X\Omega = d\tilde\eta\). Because \(d\tilde\eta\) is exact, it is 
		closed, yielding
		\begin{equation}
			\mathcal{L}_X\Omega = d(\iota_X\Omega) = d(d\tilde\eta) = 0.
		\end{equation}
		Thus, \(X\) is globally divergence-free. Near the point \(x\), $\rho \equiv 1$, which implies 
		\begin{equation}
			d\tilde\eta|_x = d\eta|_x = \beta|_x = \iota_Y\Omega|_x = \iota_v\Omega|_x.
		\end{equation}
		By the injectivity of \(\iota_\Omega\) on the fiber over \(x\), we conclude \(X(x)=v\).
	\end{proof}
	
	\subsection{Concatenation of isotopies}
	Let \(u:[0,1]\to[0,1]\) be a smooth, increasing function with \(u(t)=0\) for \(t\in[0,\delta]\) and \(u(t)=1\) for \(t\in[1-\delta,1]\) (\(\delta\in(0,1/2)\)). Define \(\lambda(t)=u(2t)\) for \(t\in[0,1/2]\) and \(\tau(t)=u(2t-1)\) for \(t\in[1/2,1]\). For two isotopies \(\Phi=\{\phi_t\},\Psi=\{\psi_t\}\), the  left concatenation is defined as,
	\[
	(\Psi\ast_l\Phi)_t=
	\begin{cases}
		\phi_{\lambda(t)}, & 0\le t\le 1/2,\\
		\psi_{\tau(t)}\circ\phi_1, & 1/2\le t\le 1.
	\end{cases}
	\]
	
	Both \(\Psi\ast_l\Phi\) and the pointwise composition \(\Psi\circ\Phi\) (where \((\Psi\circ\Phi)_t=\psi_t\circ\phi_t\)) are smooth isotopies from the identity map to \(\psi_1\circ\phi_1\). Geometrically, the left concatenation \(\Psi \ast_l \Phi\) traces the orbit of \(x\) under \(\Phi\) up to time \(1\), then continues along the orbit of the intermediate point \(\phi_1(x)\) under \(\Psi\). That is,
	\[
	\mathcal{O}_x^{\Psi \ast_l \Phi} = \mathcal{O}_x^{\Phi} \cup \mathcal{O}_{\phi_1(x)}^{\Psi},
	\] 
	where the two pieces are concatenated at the point \(\phi_1(x)\). More precisely, since \(\lambda:[0,1/2]\to[0,1]\) and \(\tau:[1/2,1]\to[0,1]\) are reparameterizations of the intervals \([0,1/2]\) and \([1/2,1]\) respectively, the parameterized orbit is:
	\[
	c_x^{\Psi \ast_l \Phi}(s) =
	\begin{cases}
		\phi_s(x), & 0 \le s \le 1 \quad (s = \lambda(t)),\\
		\psi_s(\phi_1(x)), & 0 \le s \le 1 \quad (s = \tau(t)).
	\end{cases}
	\]
	
	Interpreting \(\mathcal{F}_{\alpha}(\Phi)(1)(x)\) as 
	\[
	\mathcal{F}_{\alpha}(\Phi)(1)(x) := \int_{0}^{1}(\phi^{s})^{*}(\iota_{\dot \phi^{s}}\alpha)\,ds(x) = \int_{\mathcal{O}_x^{\Phi}}\alpha,
	\]
	it follows immediately from the above description of the orbit \(\mathcal{O}_x^{\Psi \ast_l \Phi}\) that 
	\[
	\mathcal{F}_{\alpha}(\Psi \ast_l \Phi)(1)(x) = \mathcal{F}_{\alpha}(\Phi)(1)(x) + \mathcal{F}_{\alpha}(\Psi)(1)(\phi_1(x)).
	\]
	That is, 
	\[
	\mathcal{F}_{\alpha}(\Psi \ast_l \Phi)(1) = \mathcal{F}_{\alpha}(\Phi)(1) + \mathcal{F}_{\alpha}(\Psi)(1)\circ\phi_1.
	\]
	
	\begin{proposition}\label{Conca-1}
		The paths \(\Psi\ast_l\Phi\) and \(\Psi\circ\Phi\) are homotopic relative to their endpoints.
	\end{proposition}
	
	\begin{proof}
		We construct the homotopy explicitly using separate reparameterizations for each path. Define two continuous, non-decreasing functions \(a_0, b_0: [0,1] \to [0,1]\) by:
		\[
		a_0(t) = 
		\begin{cases}
			0, & 0 \le t \le 1/2,\\
			\tau(t), & 1/2 \le t \le 1;
		\end{cases}
		\quad \text{and} \quad
		b_0(t) = 
		\begin{cases}
			\lambda(t), & 0 \le t \le 1/2,\\
			1, & 1/2 \le t \le 1.
		\end{cases}
		\]
		Since \(\tau(1/2) = u(0) = 0\) and \(\lambda(1/2) = u(1) = 1\), both \(a_0\) and \(b_0\) are continuous on \([0,1]\). For \(s \in [0,1]\), we define the linear interpolations with the identity parameter \(t\):
		\[
		a_s(t) = (1-s)a_0(t) + st, \quad \text{and} \quad b_s(t) = (1-s)b_0(t) + st.
		\]
		These functions are continuous in both \(s\) and \(t\), map \([0,1]\) to \([0,1]\), and since \(a_s(0) = b_s(0) = 0\) and \(a_s(1) = b_s(1) = 1\) for all \(s \in [0,1]\), they fix the endpoints. We now define the family of paths \(H_s \in \mathcal{P}(G_\Omega(M))\) by:
		\[
		H_s(t) = \psi_{a_s(t)} \circ \phi_{b_s(t)}, \qquad s,t \in [0,1].
		\]
		At \(s=0\), using \(\psi_0 = \mathrm{id}_M\), this gives:
		\[
		H_0(t) = 
		\begin{cases}
			\psi_0 \circ \phi_{\lambda(t)} = \phi_{\lambda(t)}, & 0 \le t \le 1/2,\\
			\psi_{\tau(t)} \circ \phi_1, & 1/2 \le t \le 1,
		\end{cases}
		\]
		which is precisely the left concatenation \(\Psi \ast_l \Phi\). At \(s=1\), we have:
		\[
		H_1(t) = \psi_t \circ \phi_t = (\Psi \circ \Phi)_t.
		\]
		
		The endpoints are fixed because:
		\[
		H_s(0) = \psi_0 \circ \phi_0 = \mathrm{id}_M, \quad \text{and} \quad H_s(1) = \psi_1 \circ \phi_1 \quad \text{for all } s \in [0,1].
		\]
		
	The homotopy $H_s$ constructed above is continuous in $(s,t)$ and smooth in $t$
	away from the corners $t=1/2$.  Since the concatenation $\ast_l$ is smooth on
	each half-interval and the transition is through the identity, a standard
	smoothing argument (see \cite[Ch.~2, \S 4]{Hirsch76}) produces a smooth homotopy
	$H_s'$ coinciding with $H_s$ near $s=0,1$ and near $t=0,1$. This completes the proof.
	\end{proof}
	
	\subsection{Classical flux homomorphism}
	According to \cite{AB1}, there is a group homomorphism $S_{\Omega}: G_{\Omega}(M) \longrightarrow H^{n-1}(M,\mathbb{R})/\Gamma_{\Omega}$ such that the following diagram commutes:
	\[
	\begin{array}{ccc}
		\widetilde{G_{\Omega}(M)} & \xrightarrow{\widetilde{S_{\Omega}}} & H^{n-1}(M,\mathbb{R}) \\
		\pi	\downarrow & & \downarrow \pi'\\
		G_{\Omega}(M) & \xrightarrow{S_{\Omega}} & H^{n-1}(M,\mathbb{R})/\Gamma_{\Omega},
	\end{array}
	\]
	where $\widetilde{G_{\Omega}(M)}$ represents the universal cover of $G_{\Omega}(M)$, and $\pi,\pi'$ are projections. Furthermore, if $\Phi = \{\phi^t\}_t \in \mathcal{P}\mathrm{Diff}^\infty_\Omega(M)$, then $\phi^1 \in \ker S_{\Omega}$ if and only if $\widetilde{S_{\Omega}}(\Phi) \in \Gamma_{\Omega}$ (see \cite{AB1}).
	
	\begin{remark}
		The following facts are well-known:
		\begin{enumerate}
			\item The subgroup $\Gamma_{\Omega}$ is discrete (unpublished result of Thurston; a proof is given in \cite{T19}).
			\item $\ker S_{\Omega}$ has the fragmentation property \cite{Thurs}.
			\item $G_{\Omega}(M)$ is locally connected by smooth arcs \cite{AB1}.
			\item $G_{\Omega}(M)$ is $p$-transitive \cite{Boo}.
		\end{enumerate}
	\end{remark}
	
	\begin{lemma}\label{L-0}(\cite{AB1})
		$\ker S_{\Omega}$ is path connected.
	\end{lemma}
	
	\begin{lemma}\label{L-01}(\cite{AB1})
		Any $\Phi \in \mathcal{P}\mathrm{Diff}^\infty_\Omega(M)$ with trivial flux is homotopic, relative to fixed endpoints, to an isotopy in $\ker S_{\Omega}$.
	\end{lemma}

	\begin{lemma}[Global section of the flux map with continuity]
		\label{lem:local_section}
		Let $(M,\Omega)$ be a closed oriented manifold with a volume form $\Omega$. There exists a globally defined, smooth map 
		\begin{equation}
			\Psi: H^{n-1}(M;\mathbb{R}) \longrightarrow \mathcal{P}\mathrm{Diff}^\infty_\Omega(M), \qquad v \longmapsto \psi_v
		\end{equation}
		which acts as a section of the flux map, satisfying:
		\begin{enumerate}
			\item $\widetilde{S}_\Omega(\psi_v) = v$ for all $v \in H^{n-1}(M;\mathbb{R})$.
			\item The map $v \mapsto \psi_v$ is continuous with respect to the compact-open ($C^0$) topology on the path space.
			\item In particular, $\psi_v \longrightarrow \mathrm{id}_M$ in the $C^0$ path space as $v \to 0$.
		\end{enumerate}
	\end{lemma}
	
	\begin{proof}
		Fix a Riemannian metric $g$ on $M$, and a basis $\{[ \eta_1 ], \dots, [ \eta_b ]\}$ of $H^{n-1}(M;\mathbb{R})$, where each $\eta_i \in \Omega^{n-1}(M)$ is a smooth harmonic form. Any $v \in H^{n-1}(M;\mathbb{R})$ can be uniquely written as $v = \sum_{i=1}^b c_i [ \eta_i ]$. We assign to each $v$ the closed $(n-1)$-form: $
		\zeta_v := \sum_{i=1}^b c_i \eta_i.$  The map $v \mapsto \zeta_v$ is linear, meaning $\zeta_v \to 0$ in the smooth $C^\infty$-topology on forms as $v \to 0$ in the Euclidean norm of $H^{n-1}(M;\mathbb{R})$. The contraction with the volume form $\Omega$ defines a vector bundle isomorphism:
		\begin{equation}
			\iota_\Omega : TM \longrightarrow \bigwedge^{n-1}T^*M, \qquad X \longmapsto i(X)\Omega.
		\end{equation}
		Since $M$ is compact, this induces a Fréchet space isomorphism on the spaces of smooth sections. Let $X_v$ be the unique smooth vector field such that $i(X_v)\Omega = \zeta_v$. Because the inverse map is continuous, the linear assignment $v \mapsto X_v$ is continuous, and $X_v \to 0$ smoothly as $v \to 0$. Furthermore, because $\zeta_v$ is closed, we have:
		\begin{equation}
			\mathcal{L}_{X_v}\Omega = d(i(X_v)\Omega) = d\zeta_v = 0,
		\end{equation}
		so $X_v$ is globally divergence-free. Since $M$ is compact, $X_v$ is integrable. We define the isotopy $\psi_v = \{\psi_v(t)\}_{t \in [0,1]}$ to be the flow of $X_v$:
		\begin{equation}
			\frac{d}{dt}\psi_v(t) = X_v \circ \psi_v(t), \qquad \psi_v(0) = \mathrm{id}_M.
		\end{equation}
		Since $X_v$ is divergence-free, its flow $\psi_v(t)$ preserves the volume form $\Omega$ for all $t$, meaning $\psi_v \in \mathcal{P}\mathrm{Diff}^\infty_\Omega(M)$. By the definition of the smooth flux homomorphism, we evaluate:
		\begin{equation}
			\widetilde{S}_\Omega(\psi_v) = \left[ \int_0^1 i(X_v)\Omega \, dt \right] = [\zeta_v] = v,
		\end{equation}
		proving that $\Psi(v) = \psi_v$ is a valid section of the flux. Finally, we address the continuity of the section. The assignment $v \mapsto X_v$ is linear on a finite-dimensional vector space, hence continuous in all $C^k-$norms. By standard ODE flow stability (Grönwall's inequality), the flow map is continuous with respect to the vector field. Explicitly, the distance any point is moved by the flow in time $t \in [0,1]$ is bounded by the supremum norm of the vector field:
		\begin{equation}
			d_g(\psi_v(t)(x), x) \le \int_0^t \|X_v(\psi_v(s)(x))\|_g \, ds \le \|X_v\|_{C^0}.
		\end{equation}
		As $v \to 0$, we have $\|X_v\|_{C^0} \to 0$, which guarantees that $\psi_v$ converges uniformly to the constant path at the identity in the compact-open topology.
	\end{proof}
	
	\section{Flux homomorphism for volume-preserving homeomorphisms}
	
	In this section we study the convergence of the sequence of real numbers $\mathcal{F}_{\alpha}(\Phi_{i})(1)(x)$ for each $x\in M$, with a view to defining and studying the flux homomorphism for isotopies of volume-preserving homeomorphisms.
	
	\begin{proposition}\label{pro1}
		Let $\alpha \in \mathcal{Z}^{1}(M)$. If $(\Phi_{i})_{i}$ is a Cauchy sequence in $\mathcal{P}\mathrm{Diff}^\infty_\Omega(M)$ in the $C^0$-metric, then the sequence of smooth functions $\mathcal{F}_{\alpha}(\Phi_{i})(1)$ is also Cauchy and the limit $\mathcal{R}(H,\alpha)\coloneqq \lim_{i\to\infty}(\mathcal{F}_{\alpha}(\Phi_{i})(1))$ exists, where $H = \lim_{C^0}(\Phi_{i})$.
	\end{proposition}
	
	\begin{proof}
		Since $(\Phi_i)_i$ is Cauchy in the $C^0$-metric (equivalently, with respect to $\bar{d}$; the two topologies coincide on $\mathcal{P}\mathrm{Diff}^\infty_\Omega(M)$), for every $\varepsilon > 0$ there exists $N$ such that for all $i,j \ge N$, we have $\bar{d}(\Phi_i,\Phi_j) \le r(g)/2$. For such $i,j$, Lemma \ref{lem:key_estimate} implies
		\begin{equation}
			\sup_{z\in M} \bigl| \mathcal{F}_{\alpha}(\Phi_i)(1)(z) - \mathcal{F}_{\alpha}(\Phi_j)(1)(z) \bigr| \le4 \|\alpha\|_\infty \, \bar{d}(\Phi_i,\Phi_j).
		\end{equation}
		Since $(\Phi_i)$ is Cauchy, the right-hand side tends to zero as $i,j \to \infty$. Thus, the sequence of continuous functions $\mathcal{F}_{\alpha}(\Phi_i)(1)$ is uniformly Cauchy and hence converges uniformly to a continuous function. Define $\mathcal{R}(H,\alpha)(z)$ to be this limit. The conclusion follows.
	\end{proof}
	
	\begin{proposition}\label{pro2}
		Let $\alpha \in \mathcal{Z}^{1}(M)$. Assume that $\Phi_i = (\{\phi_i^t\}_t)_i$ and $\Psi_i = (\{\psi_i^t\}_t)_i$ are Cauchy sequences in $\mathcal{P}\mathrm{Diff}^\infty_\Omega(M)$ with the $C^0$-metric such that $\phi_i^1 \xrightarrow{C^0} h \xleftarrow{C^0} \psi_i^1$. Then, for all $x,y\in M$ and for each piecewise smooth path from $x$ to $y$, the limits $\lim_{C^0}\int_{\gamma}\big((\phi_i^1)^*\alpha-\alpha\big)$ and $\lim_{C^0}\int_{\gamma}\big((\psi_i^1)^*\alpha-\alpha\big)$ agree.
	\end{proposition}
	
	\begin{proof}
		By Proposition \ref{pro1}, the limits exist and equal $\mathcal{R}(H,\alpha)(y)-\mathcal{R}(H,\alpha)(x)$ and $\mathcal{R}(H',\alpha)(y)-\mathcal{R}(H',\alpha)(x)$ respectively, where $H = \lim_{C^0} \Phi_i$ and $H' = \lim_{C^0} \Psi_i$. Then
		\[
		\begin{aligned}
			\Bigl|\lim_i\int_{\gamma}((\phi_i^1)^*\alpha-\alpha) - \lim_i\int_{\gamma}((\psi_i^1)^*\alpha-\alpha)\Bigr|
			&= \lim_i \Bigl|\int_{\phi_i^1\circ\gamma}\alpha - \int_{\psi_i^1\circ\gamma}\alpha\Bigr| \\
			&\le \|\alpha\|_\infty \lim_i \sup_{z\in M} d_g(\phi_i^1(z),\psi_i^1(z)) \\
			&= 0.
		\end{aligned}
		\]
		Thus the limits agree.
	\end{proof}
	
	\begin{proposition}\label{pro3}
		Let $(\Phi_i)_i$ be a sequence in $\mathcal{P}\mathrm{Diff}^\infty_\Omega(M)$ and $\alpha \in \mathcal{Z}^{1}(M)$. If $\Phi_i = \{\phi_i^t\}_t \xrightarrow{C^0} H$ with $h = H(1) := \lim_{C^0}(\phi_i^1)$, then for all $x,y\in M$ and each piecewise smooth path from $x$ to $y$ (i.e., $\gamma \in \mathcal{C}^{\mathrm{pw}}(x\to y)$), the quantity $I_{\alpha}(h,x,y)$ coincides with $\mathcal{R}(H,\alpha)(y)-\mathcal{R}(H,\alpha)(x)$.
	\end{proposition}
	
	\begin{proof}
		For each $i$, equation (\ref{Eq:intg}) gives:
		\[
		\int_{\gamma}\big((\phi_i^1)^*\alpha-\alpha\big) = \mathcal{F}_{\alpha}(\Phi_i)(1)(y) - \mathcal{F}_{\alpha}(\Phi_i)(1)(x).
		\]
		Taking the limit as $i \to \infty$ and using Proposition \ref{pro1} yields the result.
	\end{proof}
	
	The symplectic version of the following result can be found in \cite{TMH}. Its proof follows similarly by taking limits of the smooth product formula.
	
	\begin{corollary}\label{cor:product}
		Let $\Phi_i = (\{\phi_i^t\}_t)_i$ and $\Psi_i = (\{\psi_i^t\}_t)_i$ be sequences in $\mathcal{P}\mathrm{Diff}^\infty_\Omega(M)$ such that $\lim_{C^0}\Phi_i  = H:= (\{h_t\}_t)$ and $\lim_{C^0}\Psi_i = H':= (\{h'_t\}_t)$. Then for each $\alpha \in \mathcal{Z}^{1}(M)$,
		\begin{equation}
			\mathcal{R}(H\circ H', \alpha) = \mathcal{R}(H',\alpha) + \mathcal{R}(H,\alpha)\circ h'_1,
		\end{equation}
		where $\mathcal{R}(H,\alpha)\circ h'_1$ denotes the function $x \mapsto \mathcal{R}(H,\alpha)(h_1'(x))$.
	\end{corollary}
	
	\begin{proof}
		Let $(\Phi_i)_i \to H$ and $(\Psi_i)_i \to H'$ in $\bar{d}$. By Proposition \ref{Conca-1}, the paths $\Phi_i \circ \Psi_i$ and $\Phi_i \ast_l \Psi_i$ are homotopic relative to their endpoints, and since the map $\Phi \mapsto\mathcal{F}_\alpha(\Phi)(1)$ is invariant under homotopies relative to endpoints (for each closed $1$-form $\alpha$), we have
		\begin{equation}
			\mathcal{F}_\alpha(\Phi_i \circ \Psi_i)(1) = \mathcal{F}_\alpha(\Phi_i \ast_l \Psi_i)(1).
		\end{equation}
		By the smooth product formula (the explicit orbit description from Section 2), we have
		\begin{equation}
			\mathcal{F}_\alpha(\Phi_i \ast_l \Psi_i)(1) = \mathcal{F}_\alpha(\Psi_i)(1) + \mathcal{F}_\alpha(\Phi_i)(1) \circ \psi_i^1.
		\end{equation}
		Therefore,
		\begin{equation}
			\mathcal{F}_\alpha(\Phi_i \circ \Psi_i)(1) = \mathcal{F}_\alpha(\Psi_i)(1) + \mathcal{F}_\alpha(\Phi_i)(1) \circ \psi_i^1.
		\end{equation}
		Taking the limit as $i \to \infty$: by Proposition \ref{pro1}
		\begin{itemize}
			\item $\mathcal{F}_\alpha(\Phi_i \circ \Psi_i)(1) \to \mathcal{R}(H\circ H', \alpha)$;
			\item $\mathcal{F}_\alpha(\Psi_i)(1) \to \mathcal{R}(H',\alpha)$;
			\item $\mathcal{F}_\alpha(\Phi_i)(1) \circ \psi_i^1 \to \mathcal{R}(H,\alpha) \circ h'_1$ by the $C^0-$continuity of composition.
		\end{itemize}
		The result follows.
	\end{proof}
	
	\begin{lemma}\label{lem:chi_formula}
		Let $\alpha$ be a non-trivial closed $1-$form. For each $h \in \mathbb{G}^{\Omega}(M)$ and for all $z \in M$, we have
		\begin{equation}
			\chi(h,\alpha)(z) = -\Bigl( \mathcal{R}(H,\alpha)(z) - \int_M \mathcal{R}(H,\alpha)\,\Omega \Bigr),
		\end{equation}
		where $H$ is the $C^0$-limit of any Cauchy sequence $(\Phi_i)_i = (\{\phi_i^t\}_t)_i$ of isotopies in $\mathcal{P}\mathrm{Diff}^\infty_\Omega(M)$ with $h = \lim_{C^0}(\phi_i^1)$.
	\end{lemma}
	
	\begin{proof}
		From Propositions \ref{pro1} and \ref{pro2}, for all $y\in M$ and any piecewise smooth path $\gamma$ from $z$ to $y$, we have 
		
		\begin{eqnarray*} 	
			\mathcal{R}(H,\alpha)(y) - \mathcal{R}(H,\alpha)(z)  &=& \lim_{C^0}\int_{\gamma}\big((\phi_i^1)^*\alpha-\alpha\big)\\
			&=&  I_\alpha(h,z,y).
		\end{eqnarray*}
		Integrating with respect to $y$ gives:
		\begin{eqnarray*}
			\chi(h,\alpha)(z) &=& \int_M I_\alpha(h,z,y)\,\Omega\\
			&=& \int_M \mathcal{R}(H,\alpha)(y)\,\Omega - \mathcal{R}(H,\alpha)(z).
		\end{eqnarray*}
		The result follows.
	\end{proof}
	
	\begin{corollary}\label{cor2}
		Let $(\Phi_i)_i$ be a sequence in $\mathcal{P}\mathrm{Diff}^\infty_\Omega(M)$ and $\alpha \in \mathcal{Z}^{1}(M)$. If $\Phi_i \xrightarrow{C^0} H$, then the integral $\int_M \mathcal{R}(H,\alpha)\,\Omega$ is independent of the choice of any representative element in the de Rham cohomology class $[\alpha]$.
	\end{corollary}
	
	\begin{proof}
		Let $\beta \in [\alpha]$, i.e., $\alpha-\beta = df$ for some smooth function $f:M\to\mathbb{R}$. A direct computation using Stokes' theorem shows that for each $i$,
		\begin{equation}
			\int_M \mathcal{F}_\alpha(\Phi_i)(1)\,\Omega = \int_M \mathcal{F}_\beta(\Phi_i)(1)\,\Omega,
		\end{equation}
		because $\partial M = \emptyset$. Since the sequence $\mathcal{F}_\alpha(\Phi_i)(1)$ converges uniformly (by Proposition \ref{pro1}), the dominated convergence theorem applies, and taking the limit as $i \to \infty$ yields the result.
	\end{proof}
	
	\begin{corollary}\label{cor3}
		Let $\alpha$ be a closed $1$-form. If $h_1, h_2 \in \mathbb{G}^{\Omega}(M)$, then
		\begin{equation}
			\chi(h_1\circ h_2,\alpha) = \chi(h_2,\alpha) + \chi(h_1,\alpha)\circ h_2.
		\end{equation}
	\end{corollary}
	
	\begin{proof}
		This follows from the composition formula (Corollary \ref{cor:product}) and Lemma \ref{lem:chi_formula}.
	\end{proof}
	
	\subsection{Relation with the classical smooth Flux}

	For instance, each isotopy $H \in \mathcal{P}_\ast\mathbb{G}^{\Omega}(M)$ induces an element $T_H \in \Hom(H^{1}(M,\mathbb{R}),\mathbb{R})$ defined as
	\[
	T_H([\alpha]) \coloneqq \int_M \mathcal{R}(H,\alpha)\,\Omega.
	\]
	Thus, by the Poincar\'e duality theorem we have a group homomorphism
	\[
	\widetilde{L}_{\Omega}: \mathcal{P}_\ast\mathbb{G}^{\Omega}(M) \longrightarrow H^{(n-1)}(M,\mathbb{R}),
	\]
	such that $T_H([\alpha]) = \langle [\alpha], \widetilde{L}_{\Omega}(H) \rangle$, where $\langle \cdot,\cdot \rangle$ is the usual Poincaré pairing. This motivates the following factorization result which generalizes Proposition 2.4 of \cite{T19}.
	
	\begin{proposition}\label{Prop:Eq}
		Let $\alpha$ be a closed $1$-form and let $H \in \mathcal{P}_\ast\mathbb{G}^{\Omega}(M)$. Then
		\begin{equation}
			T_H([\alpha]) = \lim_{i\to\infty} \langle [\alpha], \widetilde{S}_{\Omega}(\Phi_i) \rangle,
		\end{equation}
		for any sequence $(\Phi_i)_i$ in $\mathcal{P}\mathrm{Diff}^\infty_\Omega(M)$ which converges to $H$ in the $C^0-$metric, where $\widetilde{S}_{\Omega}$ stands for the usual flux homomorphism for elements of $\mathcal{P}\mathrm{Diff}^\infty_\Omega(M)$.
	\end{proposition}
	
	\begin{proof}
		If $(\Phi_i)_i$ converges to $H$ in the $C^0$-metric, then by the uniform convergence established in Proposition \ref{pro1}, the dominated convergence theorem gives $\lim_{i\to\infty}\int_M \mathcal{F}_{\alpha}(\Phi_i)(1)\,\Omega = \int_M \mathcal{R}(H,\alpha)\,\Omega$. On the other hand, Proposition 2.4 of \cite{T19} gives $\int_M \mathcal{F}_{\alpha}(\Phi_i)(1)\,\Omega = \langle [\alpha], \widetilde{S}_{\Omega}(\Phi_i) \rangle$ for each $i$. Taking the limit yields the result.
	\end{proof}
	
	\begin{lemma}\label{lem:homotopy_invariance}
		Assume that $H, H' \in \mathcal{P}_\ast\mathbb{G}^{\Omega}(M)$ are homotopic relative to fixed endpoints. Then $\mathcal{R}(H,\alpha) = \mathcal{R}(H',\alpha)$ for each closed $1$-form $\alpha$.
	\end{lemma}
	
	\begin{proof}
		Let $\Sigma : [0,1] \times [0,1] \longrightarrow \mathbb{G}^{\Omega}(M)$ be the homotopy between $H$ and $H'$ such that
		\[
		\Sigma(0,t) = h_t,\qquad \Sigma(1,t) = h'_t,
		\]
		and with fixed endpoints:
		\[
		\Sigma(s,0) = \mathrm{id}_M,\qquad \Sigma(s,1) = h,
		\]
		for all $s\in[0,1]$, where $h = H(1) = H'(1)$. By definition of $\mathcal{P}_\ast\mathbb{G}^{\Omega}(M)$, there exist sequences of smooth volume-preserving isotopies $\Phi_i = \{\phi_t^i\}_t$ and $\Psi_i = \{\psi_t^i\}_t$ such that
		\[
		H = \lim_{C^0} \Phi_i, \qquad H' = \lim_{C^0} \Psi_i.
		\]
		For $i$ sufficiently large, we may assume
		\[
		\bar{d}(H,\Phi_i) < \frac{r(g)}{2}, \qquad \bar{d}(H',\Psi_i) < \frac{r(g)}{2},
		\]
		where $r(g)$ is the injectivity radius of $(M,g)$. Hence, for each $x\in M$ and $t\in[0,1]$, the points $\phi_t^i(x)$ and $h_t(x)$ (resp. $\psi_t^i(x)$ and $h'_t(x)$) are connected by a unique minimal geodesic $\lambda_{i,x}^{s,t}$ (resp. $\mu_{i,x}^{s,t}$). For each fixed $x\in M$, consider the $2-$chain formed by:
		\begin{itemize}
			\item the cylinder $A_1$ traced by the geodesics $\lambda_{i,x}^{s,t}$ connecting $\mathcal{O}_x^{\Phi_i}$ to $\mathcal{O}_x^{H}$;
			\item the homotopy square $\square_{H,H'}^x = \{ \Sigma(s,t)(x) : (s,t)\in[0,1]^2 \}$;
			\item the cylinder $A_2$ traced by the geodesics $\mu_{i,x}^{s,t}$ connecting $\mathcal{O}_x^{H'}$ to $\mathcal{O}_x^{\Psi_i}$.
		\end{itemize}
	 The continuous orbits
		\(\mathcal{O}_x^{H}\) and \(\mathcal{O}_x^{H'}\) are uniform limits of these
		smooth paths. By the standard approximation theorem for continuous paths,
		we may approximate each continuous orbit \(\mathcal{O}_x^{H}\) and
		\(\mathcal{O}_x^{H'}\) by piecewise smooth paths with the same endpoints,
		uniformly to within \(\varepsilon>0\). Assume this done and denote the piecewise smooth paths \(\mathcal{O}_x^{H}\) and
		\(\mathcal{O}_x^{H'}\).	Because $H$ and $H'$ have the same endpoints, this union is a closed $2-$chain. Its boundary is the piecewise smooth closed curve
		\[
		\mathcal{O}_x^{\Phi_i} \;\cup\; \lambda_{i,x}^{s,1} \;\cup\; \mathcal{O}_x^{\Psi_i} \;\cup\; \mu_{i,x}^{s,1}.
		\]
		Applying Stokes' theorem to any closed $1-$form $\alpha$ gives
		\[
		\int_{\mathcal{O}_x^{\Phi_i}} \alpha + \int_{\lambda_{i,x}^{s,1}} \alpha
		=
		\int_{\mathcal{O}_x^{\Psi_i}} \alpha + \int_{\mu_{i,x}^{s,1}} \alpha.
		\]
		Thus, 
		\[
		\begin{aligned}
			\left| \mathcal{F}_\alpha(\Phi_i)(1)(x) - \mathcal{F}_\alpha(\Psi_i)(1)(x) \right|
			&= \left| \int_{\mathcal{O}_x^{\Phi_i}} \alpha - \int_{\mathcal{O}_x^{\Psi_i}} \alpha \right| \\
			&= \left| \int_{\lambda_{i,x}^{s,1}} \alpha - \int_{\mu_{i,x}^{s,1}} \alpha \right| \\
			&\le \|\alpha\|_\infty \left( \operatorname{length}(\lambda_{i,x}^{s,1}) + \operatorname{length}(\mu_{i,x}^{s,1}) \right) \\
			&\le \|\alpha\|_\infty \left( \bar{d}(H,\Phi_i) + \bar{d}(H',\Psi_i) \right),
		\end{aligned}
		\]
		where the last inequality follows since the geodesic lengths are bounded by the $C^0$-distance between the curves. Taking the supremum over $x\in M$ and then the limit as $i\to\infty$, we obtain $
		\mathcal{R}(H,\alpha) = \mathcal{R}(H',\alpha),$ 
		for every closed $1$-form $\alpha$.
	\end{proof}

	\begin{lemma}\label{lem:null_homologous}
		Assume that $H, H' \in \mathcal{P}_\ast\mathbb{G}^{\Omega}(M)$ have the same
		endpoints $h = H(1) = H'(1)$.  If there exists a point $z\in M$ such that
		each of the orbits $\mathcal{O}_z^{H}$ and $\mathcal{O}_z^{H'}$ is piecewise
		smooth and together they bound a null-homologous $2-$chain, then
		$\widetilde{L}_{\Omega}(H) = \widetilde{L}_{\Omega}(H')$.
	\end{lemma}
	
	\begin{proof}
		Let $\alpha$ be any closed $1$-form on $M$.  For an arbitrary point $x\in M$,
		Proposition~\ref{pro3} gives
		\[
		\mathcal{R}(H,\alpha)(x) - \mathcal{R}(H,\alpha)(z)
		= I_\alpha(h,z,x),
		\qquad
		\mathcal{R}(H',\alpha)(x) - \mathcal{R}(H',\alpha)(z)
		= I_\alpha(h,z,x),
		\]
		where $I_\alpha(h,\cdot,\cdot)$ is the well-defined function from
		\eqref{Eq:intgg} that depends only on the endpoint $h$, not on the particular
		isotopy (Proposition~\ref{pro2}).  Subtracting the two equalities yields
		\begin{equation}
			\bigl(\mathcal{R}(H,\alpha) - \mathcal{R}(H',\alpha)\bigr)(x)
			= \bigl(\mathcal{R}(H,\alpha) - \mathcal{R}(H',\alpha)\bigr)(z)
			\qquad \forall x\in M.
		\end{equation}
		Hence the difference $D(x) := \mathcal{R}(H,\alpha)(x) - \mathcal{R}(H',\alpha)(x)$
		is a constant function on $M$. Now evaluate $D$ at the special point $z$.  By definition,
		$\mathcal{R}(H,\alpha)(z) = \lim_{i\to\infty} \int_{\mathcal{O}_z^{\Phi_i}} \alpha$
		and similarly for $H'$, where $(\Phi_i)_i$, $(\Psi_i)_i$ are smooth volume-preserving
		isotopies converging uniformly to $H$ and $H'$, respectively.  Because the
		isotopies converge uniformly, the smooth paths $\phi_i^t(z)$ converge uniformly to
		$h^t(z)$, and the same holds for the other sequence.  Since $\alpha$ is smooth,
		the line integral along a continuous path depends continuously on the path
		with respect to uniform convergence; therefore
		\begin{equation}
			\int_{\mathcal{O}_z^{\Phi_i}} \alpha \;\longrightarrow\;
			\int_{\mathcal{O}_z^{H}} \alpha,
			\qquad
			\int_{\mathcal{O}_z^{\Psi_i}} \alpha \;\longrightarrow\;
			\int_{\mathcal{O}_z^{H'}} \alpha,
		\end{equation}
		where the right-hand sides are classical line integrals along piecewise smooth
		curves.  Consequently,
		\begin{equation}
			D(z) = \int_{\mathcal{O}_z^{H}} \alpha \;-\; \int_{\mathcal{O}_z^{H'}} \alpha.
		\end{equation}
		The hypothesis that the two piecewise smooth orbits bound a null-homologous
		$2-$chain implies, by Stokes' theorem for piecewise smooth singular chains,
		that the difference of the integrals vanishes:
		\begin{equation}
			\int_{\mathcal{O}_z^{H}} \alpha \;-\; \int_{\mathcal{O}_z^{H'}} \alpha = 0.
		\end{equation}
		Thus, $D(z)=0$ and because $D$ is constant we have $D(x)=0$ for all $x$. Finally, integrating over $M$ with respect to the normalized volume form
		$\Omega$,
		\begin{equation}
			0 = \int_M D\,\Omega
			= \int_M \mathcal{R}(H,\alpha)\,\Omega - \int_M \mathcal{R}(H',\alpha)\,\Omega
			= \langle [\alpha],\, \widetilde{L}_{\Omega}(H) - \widetilde{L}_{\Omega}(H') \rangle.
		\end{equation}
		Since this holds for every closed $1-$form $\alpha$, Poincar\'e duality forces
		$\widetilde{L}_{\Omega}(H) = \widetilde{L}_{\Omega}(H')$.
	\end{proof}
	
	\begin{lemma}\label{lem0}
		Let $H \in \mathcal{P}_\ast\mathbb{G}^{\Omega}(M)$ be a loop at the identity map
		(i.e., $H(1)=\mathrm{id}_M$).
		\begin{enumerate}
			\item For each closed $1-$form $\alpha$, the continuous function
			$x\mapsto \mathcal{R}(H,\alpha)(x)$ is constant and equals
			$\langle [\alpha], \widetilde{L}_{\Omega}(H)\rangle$.
			\item If there exists a point $z\in M$ for which the orbit
			$\mathcal{O}_z^{H}$ is piecewise smooth and bounds a null-homologous
			$2-$chain in $M$, then $\widetilde{L}_{\Omega}(H)=0$.
		\end{enumerate}
	\end{lemma}
	
	\begin{proof}
		\textbf{(1)}  Since $H$ is a loop, $H(1)=\mathrm{id}_M$.
		Let $x,y\in M$ and a piecewise smooth path
		$\gamma$ from $x$ to $y$.  By Proposition~\ref{pro3},
		\begin{equation}
			\mathcal{R}(H,\alpha)(y)-\mathcal{R}(H,\alpha)(x)
			= I_\alpha(\mathrm{id}_M,x,y)
			= \lim_{i\to\infty} \int_\gamma \bigl((\phi_i^1)^*\alpha - \alpha\bigr),
		\end{equation}
		where $(\Phi_i)_i = (\{\phi_i^t\}_t)_i$ is any sequence of smooth
		volume-preserving isotopies converging to $H$ in the $C^0$-metric.
		For each $i$,
		\begin{equation}
			\int_\gamma \bigl((\phi_i^1)^*\alpha - \alpha\bigr)
			= \int_{\phi_i^1\circ\gamma} \alpha \;-\; \int_\gamma \alpha .
		\end{equation}
		Because $\phi_i^1 \to \mathrm{id}_M$ uniformly on $M$, the paths
		$\phi_i^1\circ\gamma$ converge uniformly to $\gamma$.  The line integral
		of a smooth $1$-form along a continuous path depends continuously on the
		path with respect to uniform convergence; hence
		$\int_{\phi_i^1\circ\gamma}\alpha \to \int_\gamma\alpha$, and the
		difference converges to $0$.  Therefore $\mathcal{R}(H,\alpha)(y)=
		\mathcal{R}(H,\alpha)(x)$, proving that $\mathcal{R}(H,\alpha)$ is constant. By definition, $\widetilde{L}_\Omega(H)$ satisfies
		$\langle [\alpha], \widetilde{L}_\Omega(H)\rangle
		= \int_M \mathcal{R}(H,\alpha)\,\Omega$.  Since $\mathcal{R}(H,\alpha)$
		is constant, this integral equals the constant value multiplied by the
		total volume (which is $1$).  Thus the constant equals
		$\langle [\alpha], \widetilde{L}_\Omega(H)\rangle$.
		
		\textbf{(2)}  Suppose there exists $z\in M$ such that the orbit
		$\mathcal{O}_z^{H}$ is a piecewise smooth loop and bounds a
		null-homologous $2-$chain.  Let $H'$ be the constant loop at the
		identity (so $H'(t)=\mathrm{id}_M$ for all $t$).  Its orbit
		$\mathcal{O}_z^{H'}$ is the point $\{z\}$, which is trivially piecewise
		smooth.  The two orbits together bound the same null-homologous
		$2-$chain (the one given by $\mathcal{O}_z^{H}$ together with the
		degenerate point).  Thus Lemma~\ref{lem:null_homologous} applies and
		yields $\widetilde{L}_\Omega(H) = \widetilde{L}_\Omega(H')$.  The
		flux of the constant loop is obviously $0$, so
		$\widetilde{L}_\Omega(H)=0$.
	\end{proof}
	
	\begin{lemma}\label{lem:homotopy_invariance_flux}
		If $H, H' \in \mathcal{P}_\ast\mathbb{G}^{\Omega}(M)$ are homotopic relative to fixed endpoints, then 
		\begin{equation}
			\widetilde{L}_{\Omega}(H) = \widetilde{L}_{\Omega}(H').
		\end{equation}

	\end{lemma}
	
	\begin{proof}
		By Lemma~\ref{lem:homotopy_invariance}, we have 
		$
		\mathcal{R}(H,\alpha)=\mathcal{R}(H',\alpha)
		$
		for every closed $1$-form $\alpha$. Integrating over $M$ with respect to
		$\Omega$, we get
		$
		\widetilde{L}_\Omega(H)=\widetilde{L}_\Omega(H').
		$
	\end{proof}

\subsection{An intrinsic definition of topological flux}
\label{sec:intrinsic-flux}

For $H \in \mathcal{P}_\ast\mathbb{G}^{\Omega}(M)$, the topological flux $\widetilde{L}_\Omega(H)$ was defined in
Proposition~\ref{Prop:Eq} as the limit of the smooth fluxes
$\widetilde{S}_\Omega(\Phi_i)$ along any sequence $(\Phi_i)_i \subset \mathcal{P}\mathrm{Diff}^\infty_\Omega(M)$
converging to $H$ in the $C^0$-metric.  We now observe that this limit can be written intrinsically.

\begin{definition}
	Let $H \in \mathcal{P}_\ast\mathbb{G}^{\Omega}(M)$ and $\alpha \in \mathcal{Z}^1(M)$.  We define the
	\emph{intrinsic continuous flux} of $H$ evaluated at $[\alpha]$ by
	\begin{equation}
		\widetilde{S}_{C^0}(H)([\alpha]) 
		\coloneqq \int_M \mathcal{R}(H,\alpha)\,\Omega,
	\end{equation}
	where $\mathcal{R}(H,\alpha)$ is the continuous function defined in \eqref{eq:Rdef}.
	Equivalently, for any sequence $(\Phi_i)_i \subset \mathcal{P}\mathrm{Diff}^\infty_\Omega(M)$ with
	$\Phi_i \to H$ in $\bar d$,
	\begin{equation}
		\widetilde{S}_{C^0}(H)([\alpha]) 
		= \lim_{i\to\infty} \widetilde{S}_\Omega(\Phi_i)([\alpha]).
	\end{equation}
\end{definition}

\begin{proposition}
	The map $\widetilde{S}_{C^0}$ is well-defined, depends only on the cohomology
	class $[\alpha]$, and coincides with $\widetilde{L}_\Omega$ on
	$\mathcal{P}_\ast\mathbb{G}^{\Omega}(M)$.
\end{proposition}

\begin{proof}
	If $\alpha = df$ is exact, then for each smooth $\Phi_i$,
	\[
	\widetilde{S}_\Omega(\Phi_i)([df]) = \int_M (f\circ \phi_i^1 - f)\,\Omega = 0,
	\]
	because $\phi_i^1$ preserves $\Omega$.  Passing to the limit as $i\to\infty$ gives
	$\widetilde{S}_{C^0}(H)([df]) = 0$.  Hence the map depends only on $[\alpha]$.	The equality $\widetilde{S}_{C^0} = \widetilde{L}_\Omega$ follows from
	Proposition~\ref{Prop:Eq} and the definition of $\mathcal{R}$.
\end{proof}
\begin{corollary}[Poincar\'e dual characterization of the topological flux]
	\label{cor:flux_pairing}
	Let \(H\in \mathcal P_*\mathbb G^\Omega(M)\). Then, there exists a unique class
	$
	\widetilde L_\Omega(H)\in H^{n-1}(M;\mathbb R)
	$
	such that for every closed \(1\)-form \(\alpha\),
	\[
	\int_M \mathcal R(H,\alpha)\,\Omega
	=
	\langle [\alpha], \widetilde L_\Omega(H)\rangle .
	\]
\end{corollary}

\begin{proof}
	The assignment
	$
	[\alpha]\longmapsto \int_M \mathcal R(H,\alpha)\,\Omega
	$
	is a well-defined linear functional on \(H^1(M;\mathbb R)\). Indeed, if
	\(\alpha=df\) is exact, then \(\mathcal R(H,df)\) is the \(C^0\)-limit of
	\(f\circ\phi_i^1-f\), and since each \(\phi_i^1\) preserves \(\Omega\),
	$
	\int_M \bigl(f\circ\phi_i^1-f\bigr)\,\Omega=0.
	$
	Passing to the limit gives
	$
	\int_M \mathcal R(H,df)\,\Omega=0.
	$
	By Poincar\'e duality,
	\[
	H^1(M;\mathbb R)^*\cong H^{n-1}(M;\mathbb R),
	\]
	so this linear functional is represented uniquely by a class
	\(\widetilde L_\Omega(H)\in H^{n-1}(M;\mathbb R)\).
\end{proof}
	
	\begin{proposition}\label{pro:dualty}
		The map $\widetilde{L}_{\Omega}$ is the Poincar\'e dual of Fathi's mass flow.
	\end{proposition}

	\begin{proof}
		Let $\mu$ be the normalized measure induced by $\Omega$. Fathi's mass flow theorem gives a continuous homomorphism
		$
		\tilde{\mathfrak{F}} : \widetilde{\Homeo}_0(M,\mu) \longrightarrow H_1(M;\mathbb{R}).
		$
		Passing from a continuous isotopy $H\in\mathcal P_*\mathbb G^\Omega(M)$ to its class $[H]$ in the universal cover $\widetilde{\Homeo}^\Omega_0(M)$, we obtain the composition
		$
		\mathcal{PD}\circ\tilde{\mathfrak{F}}\circ\pi,
		$
		where
		$
		\pi:\mathcal P_*\mathbb G^\Omega(M)\longrightarrow \widetilde{\Homeo}^\Omega_0(M)
		$
		is the natural projection.	For smooth isotopies, Fathi proves in Appendix A.5 of \cite{Fathi1980} that
		$
		\widetilde{S}_\Omega(\Phi)
		=
		\mathcal{PD}\bigl(\tilde{\mathfrak{F}}([\Phi])\bigr)
		$
		for every smooth volume-preserving isotopy $\Phi\in\mathcal P\mathrm{Diff}^\infty_\Omega(M)$. 
		Both maps
		$
		\widetilde{L}_\Omega
		\quad\text{and}\quad
		\mathcal{PD}\circ\tilde{\mathfrak{F}}\circ\pi
		$
		are continuous. Since smooth volume-preserving isotopies are dense in $\mathcal P_*\mathbb G^\Omega(M)$, the identity, 
		$
		\widetilde{L}_{\Omega}(H)
		=
		\mathcal{PD}\bigl(\tilde{\mathfrak{F}}([H])\bigr)
		$
		extends to all continuous isotopies $H\in\mathcal P_*\mathbb G^\Omega(M)$. Therefore,
		$
		\widetilde{L}_\Omega
		=
		\mathcal{PD}\circ\tilde{\mathfrak{F}}\circ\pi.
		$
	\end{proof}
	\subsection{The $C^0$-rigidity of the flux lattice}
	
	Put
	\begin{equation}\label{rigi-3}
		\widetilde{\Gamma}_{\Omega} 
		:= \widetilde{S}_{C^0}\Bigl(
		\pi_1^{\mathrm{path}}\bigl(\mathbb{G}^{\Omega}(M)\bigr)
		\Bigr).
	\end{equation}
	We equip \(\widetilde{\Gamma}_{\Omega} \subset H^{n-1}(M,\mathbb{R})\) with 
	the natural topology. There is a group homomorphism
	\begin{equation}\label{rigi-4}
		S_{C^0}: \mathbb{G}^{\Omega}(M) \longrightarrow 
		H^{n-1}(M,\mathbb{R})/\widetilde{\Gamma}_{\Omega},
	\end{equation}
	such that the following diagram commutes:
	\[
	\begin{array}{ccc}
		\widetilde{\mathbb{G}^{\Omega}(M)} & 
		\xrightarrow{\ \widetilde{S}_{C^0}\ } & H^{n-1}(M,\mathbb{R}) \\
		\mathrm{ev}_1 \downarrow & & \downarrow \pi' \\
		\mathbb{G}^{\Omega}(M) & 
		\xrightarrow{\ S_{C^0}\ } & 
		H^{n-1}(M,\mathbb{R})/\widetilde{\Gamma}_{\Omega}.
	\end{array}
	\]
	Here \(\mathrm{ev}_1\) denotes the endpoint evaluation 
	\([H] \mapsto H(1)\).

	\begin{proposition}\label{P4.4}
	Let $H = \{h^t\}_t \in \mathcal{P}_\ast\mathbb{G}^{\Omega}(M)$. Then $h^1 \in \ker S_{C^0}$ if and only if $\widetilde{S}_{C^0}(H) \in \widetilde{\Gamma}_{\Omega}$.
\end{proposition}

\begin{proof}
	If $h^1 \in \ker S_{C^0}$, then by the definition of $S_{C^0}$, the image of the equivalence class of $H$ under $\pi' \circ \widetilde{S}_{C^0}$ is zero in the quotient. Hence $\widetilde{S}_{C^0}(H) \in \ker \pi' = \widetilde{\Gamma}_{\Omega}$.
	Conversely, if $\widetilde{S}_{C^0}(H) \in \widetilde{\Gamma}_{\Omega}$, then $\pi'(\widetilde{S}_{C^0}(H)) = 0$ in $H^{n-1}(M,\mathbb{R})/\widetilde{\Gamma}_{\Omega}$. By commutativity of the diagram, $S_{C^0}(h^1) = 0$, so $h^1 \in \ker S_{C^0}$.
\end{proof}

\begin{theorem}[$C^0$-Rigidity of the flux group]
	\label{thm:C0rigidity}
	For any closed oriented manifold $M$ of dimension $n\ge 2$,
	\[
	\widetilde S_{C^0}\bigl(\pi_1^{\mathrm{path}}(\mathbb{G}^{\Omega}(M))\bigr)
	=
	\Gamma_{\Omega}.
	\]
	Consequently, the topological flux group is discrete and has rank $b_1(M)$.
\end{theorem}

\begin{proof}
	Let $\mu$ be the normalized measure induced by $\Omega$. We prove the two
	inclusions.
	
	\smallskip
	\noindent\textbf{Inclusion $\supseteq$.}
	Let $v\in\Gamma_\Omega$. By definition of the smooth flux group, there
	exists a smooth loop $\Phi=\{\phi^t\}_{t\in[0,1]}\in
	\mathcal{P}\mathrm{Diff}^\infty_\Omega(M)$ with
	$\phi^0=\phi^1=\mathrm{id}_M$ and $\widetilde S_\Omega(\Phi)=v$.
	Since $G_\Omega(M)\subset\mathbb{G}^{\Omega}(M)$ and the topological flux
	agrees with the smooth flux on smooth isotopies (Proposition~\ref{Prop:Eq}),
	the loop $\Phi$ represents an element of
	$\pi_1^{\mathrm{path}}(\mathbb{G}^{\Omega}(M))$ whose topological flux is
	$v$. Hence
	\[
	\Gamma_\Omega
	\;\subseteq\;
	\widetilde S_{C^0}\bigl(\pi_1^{\mathrm{path}}(\mathbb{G}^{\Omega}(M))\bigr).
	\]
	
	\smallskip
	\noindent\textbf{Inclusion $\subseteq$.}
	Let $H\in\mathcal L_*\mathbb{G}^{\Omega}(M)$ be a loop at the identity,
	so that $[H]\in\pi_1(\operatorname{Homeo}_0(M,\mu))$. By
	Proposition~\ref{pro:dualty}, the intrinsic topological flux is the
	Poincar\'e dual of Fathi's mass flow:
	\begin{equation}
		\widetilde S_{C^0}(H)
		=
		\mathcal{PD}\bigl(\tilde{\mathfrak F}([H])\bigr),
	\end{equation}
	where $\mathcal{PD}:H_1(M;\mathbb R)\to H^{n-1}(M;\mathbb R)$ denotes
	Poincar\'e duality. By Fathi's theorem
	\cite[Corollaire~6.3]{Fathi1980},
	\[
	\tilde{\mathfrak F}\bigl(\pi_1(\operatorname{Homeo}_0(M,\mu))\bigr)
	=
	H_1(M;\mathbb Z)_{\mathrm{free}}.
	\]
	Hence
	\[
	\widetilde S_{C^0}(H)
	\;\in\;
	\mathcal{PD}\bigl(H_1(M;\mathbb Z)_{\mathrm{free}}\bigr).
	\]
	By Banyaga's theorem \cite[Ch.~5]{AB1}, the smooth flux group of a closed
	oriented manifold with volume form coincides with the Poincar\'e dual of
	the free part of integral homology:
	\[
	\Gamma_\Omega
	=
	\mathcal{PD}\bigl(H_1(M;\mathbb Z)_{\mathrm{free}}\bigr).
	\]
	Consequently $\widetilde S_{C^0}(H)\in\Gamma_\Omega$. Since $H$ was an
	arbitrary loop, we obtain
	\[
	\widetilde S_{C^0}\bigl(\pi_1^{\mathrm{path}}(\mathbb{G}^{\Omega}(M))\bigr)
	\;\subseteq\;
	\Gamma_\Omega.
	\]
	
	\smallskip
	Combining both inclusions yields the announced equality. Discreteness and
	rank follow from Banyaga's theorem \cite[Ch.~5]{AB1}, which states that
	$\Gamma_\Omega$ is a discrete lattice in $H^{n-1}(M;\mathbb R)$ of rank
	$b_1(M)$.
\end{proof}
	
	The flux obstructs the existence of fixed points and prevents uniform convergence to the identity. If a homeomorphism $h \in \mathbb{G}^{\Omega}(M)$ satisfies $S_{C^0}(h) \neq 0$, then $h$ cannot be accumulated by homeomorphisms with trivial flux; in particular, on manifolds like the torus, non-zero flux implies the possible absence of fixed points. Conversely, every element of $\ker S_{C^0}$ on a torus must possess at least two fixed points (Franks' theorem), giving a clean topological dichotomy \cite{Franks1988}.
	
	\begin{lemma}\label{L-02}
		Let $H=\{h^t\}_{t\in[0,1]}$ be a continuous volume-preserving isotopy in $ \mathcal{P}_\ast\mathbb{G}^{\Omega}(M)$ 
		such that $
		h^t\in\ker S_{C^0}$ 
		for every $t\in[0,1].$ 
		Then,  $H$ is a topological vanishing-flux isotopy, i.e.
		$
		\widetilde{S}_{C^0}(H)=0.
		$
	\end{lemma}
	
	\begin{proof}
		For each fixed $t$, define the restricted isotopy
		$
		Q_t(s)=h^{st},$ for all $s\in[0,1].
		$
		Its endpoint is $h^t$. Since $h^t\in\ker S_{C^0}$ by assumption,
		Proposition~\ref{P4.4} gives
		$
		\widetilde{S}_{C^0}(Q_t)\in\widetilde{\Gamma}_\Omega.
		$
		The map, 
		$
		t\longmapsto \widetilde{S}_{C^0}(Q_t)
		$
		is continuous, and its image lies in the discrete subgroup
		$\widetilde{\Gamma}_\Omega$. Hence, it is locally constant, therefore constant
		on $[0,1]$. At $t=0$, $Q_0$ is the constant isotopy at the identity, so $
		\widetilde{S}_{C^0}(Q_0)=0.$ 
		Thus, $
		\widetilde{S}_{C^0}(Q_t)=0$
		for all $t\in[0,1].
		$
		Taking $t=1$ gives $
		\widetilde{S}_{C^0}(H)=0.$ 
	\end{proof}

\begin{lemma}[Homotopy to the kernel of the flux]\label{L-03}
	Let \((M,\Omega)\) be a closed oriented manifold equipped with a volume form.
	If \(H \in \mathcal{P}_*\mathbb{G}^{\Omega}(M)\) has flux
	$
	\widetilde{S}_{C^0}(H)=0,
	$
	then \(H\) is homotopic relative to its endpoints to an isotopy
	\(G=\{g^t\}_{t\in[0,1]}\in\mathcal{P}_*\mathbb{G}^{\Omega}(M)\)
	such that
	$
	g^t\in\ker S_{C^0}$  for all $t\in[0,1].
	$
\end{lemma}

\begin{proof}
	Recall that,
	$
	\widetilde{\mathbb{G}^{\Omega}(M)}
	:=
	\mathcal P_*\mathbb{G}^{\Omega}(M)\big/\sim,
	$
	where two isotopies are equivalent if they are homotopic relative to fixed
	endpoints inside \(\mathcal P_*\mathbb{G}^{\Omega}(M)\). This is a group under
	pointwise composition:
	$
	(H\cdot K)(t):=H(t)\circ K(t).
	$
	Indeed, the class of \(H\cdot K\) depends only on the classes of \(H\) and \(K\).
	By Lemma~\ref{lem:homotopy_invariance}, the topological flux induces  a
	well-defined continuous homomorphism
	\[
	\widetilde S_{C^0}:	\mathcal P_*\mathbb{G}^{\Omega}(M)
	\longrightarrow H^{n-1}(M;\mathbb R).
	\]
	
	By Lemma~\ref{lem:local_section}, there is a continuous section
	\[
	\Psi:H^{n-1}(M;\mathbb R)\longrightarrow \mathcal P\mathrm{Diff}^\infty_\Omega(M),
	\qquad
	\widetilde S_\Omega(\Psi(v))=v.
	\]
	Passing to homotopy classes relative to endpoints gives
	\[
	\overline\Psi:H^{n-1}(M;\mathbb R)\longrightarrow
		\mathcal P_*\mathbb{G}^{\Omega}(M),
	\qquad
	\overline\Psi(v):=[\Psi(v)],
	\]
	which still satisfies
	$
	\widetilde S_{C^0}\bigl(\overline\Psi(v)\bigr)=v.
	$
	We write \(\Psi\) for \(\overline\Psi\).	Thus, every \(x\in 	\mathcal P_*\mathbb{G}^{\Omega}(M)\) can be written uniquely as
	\[
	x=\Psi(u)\,k,
	\qquad
	u:=\widetilde S_{C^0}(x),
	\qquad
	k:=\Psi(u)^{-1}x\in\ker\widetilde S_{C^0}.
	\]
	
	Now, let \(H\in\mathcal P_*\mathbb{G}^{\Omega}(M)\) with
	\(\widetilde S_{C^0}(H)=0\). Define, 
	\[
	\widetilde H:[0,1]\longrightarrow 	\mathcal P_*\mathbb{G}^{\Omega}(M),
	\qquad
	\widetilde H(t):=H|_{[0,t]},
	\]
	and set
	$
	v(t):=\widetilde S_{C^0}\bigl(\widetilde H(t)\bigr).
	$
	Then \(v:[0,1]\to H^{n-1}(M;\mathbb R)\) is continuous,
	\(v(0)=0\), and \(v(1)=\widetilde S_{C^0}(H)=0\).	Write
	\[
	\widetilde H(t)=\Psi(v(t))\,k(t),
	\qquad
	k(t)\in\ker\widetilde S_{C^0}.
	\]
	Since \(v(0)=v(1)=0\) and \(\Psi(0)=\mathrm{id}\), we have
	$
	k(0)=\mathrm{id},$ and 
	$
	k(1)=\widetilde H(1).$ 
	Because \(H^{n-1}(M;\mathbb R)\) is a vector space, it is contractible.
	For \(s\in[0,1]\), define
	\[
	\widetilde H_s(t)
	:=
	\Psi\bigl((1-s)v(t)\bigr)\,k(t).
	\]
	For each fixed \(s\), the map \(t\mapsto \widetilde H_s(t)\) is continuous
	in \(	\mathcal P_*\mathbb{G}^{\Omega}(M) \), and we have 
	$
	\widetilde H_s(0)=\mathrm{id},
	$ as well $
	\widetilde H_s(1)=\widetilde H(1).$ 
	Thus, \(\{\widetilde H_s\}_{s\in[0,1]}\) is a homotopy relative to endpoints
	between \(\widetilde H_0=\widetilde H\) and \(\widetilde H_1=k\).	Let
	\[
	ev_1:	\mathcal P_*\mathbb{G}^{\Omega}(M) \longrightarrow \mathbb{G}^{\Omega}(M),
	\qquad
		ev_1([L]):=L(1),
	\]
	be the endpoint evaluation. Set, $
	G_s:= 	ev_1\circ \widetilde H_s, $ and derive that  
	 \(\{G_s\}_{s\in[0,1]}\) is a continuous family of isotopies in
	\(\mathcal P_*\mathbb{G}^{\Omega}(M)\) satisfying
	\[
	G_0=H,
	\qquad
	G_s(0)=\mathrm{id}_M,
	\qquad
	G_s(1)=H(1)\quad\forall s.
	\]
	At \(s=1\), we have \(\widetilde H_1(t)=k(t)\), so
	\[
	\widetilde S_{C^0}\bigl(\widetilde H_1(t)\bigr)=0
	\qquad\forall t\in[0,1].
	\]
	Hence \(G_1(t)= 	ev_1(k(t))\in\ker S_{C^0}\) for every \(t\). Therefore
	\(G:=G_1\) is the required isotopy.
\end{proof}

\begin{proposition}\label{Connect}
	The subgroup $\ker S_{C^0}$ is path connected.
\end{proposition}

\begin{proof}
	This proof follows immediately from Lemma~\ref{L-03}. Let $h \in \ker S_{C^0}$ be arbitrary. We must show that there exists a continuous path entirely within $\ker S_{C^0}$ connecting the identity to $h$.  By the definition of the path-closure group $\mathbb{G}^{\Omega}(M)$, it is path connected, so we can choose a continuous path $H \in \mathcal{P}_\ast\mathbb{G}^{\Omega}(M)$ from the identity to $h$. Let 
	\[
	v := \widetilde{S}_{C^0}(H) \in H^{n-1}(M;\mathbb{R}),
	\]
	be its total flux. Since $h \in \ker S_{C^0}$, the definition of the quotient map $S_{C^0} : \mathbb{G}^{\Omega}(M) \to H^{n-1}(M;\mathbb{R})/\widetilde{\Gamma}_\Omega$ implies that $v \in \widetilde{\Gamma}_\Omega$. Because $\widetilde{\Gamma}_\Omega = \widetilde{S}_{C^0}\bigl( 	\pi_1^{\mathrm{path}}\bigl(\mathbb{G}^{\Omega}(M)\bigr) \bigr)$, there exists a loop $\gamma \in 	\pi_1^{\mathrm{path}}\bigl(\mathbb{G}^{\Omega}(M)\bigr)$ based at the identity such that $
	\widetilde{S}_{C^0}(\gamma) = -v.$ 
	Define a new continuous path $H' := \gamma \ast H$ as the concatenation of $\gamma$ followed by $H$. Since $\gamma$ is a loop at the identity, $H'$ is still a well-defined continuous path from the identity to $h$. Its total flux is the sum of the fluxes:
	\[
	\widetilde{S}_{C^0}(H') = \widetilde{S}_{C^0}(\gamma) + \widetilde{S}_{C^0}(H) = -v + v = 0.
	\]
	
	Thus, $H'$ is a path from the identity to $h$ with trivial total topological flux. By Lemma~\ref{L-03}, $H'$ is homotopic relative to its endpoints to a continuous path $P \in \mathcal{P}_\ast\mathbb{G}^{\Omega}(M)$ such that $P(t) \in \ker S_{C^0}$ for every $t \in [0,1]$. 	Since the homotopy fixes the endpoints, $P(0) = \mathrm{id}$ and $P(1) = h$. Therefore, $t \mapsto P(t)$ is the required continuous path in $\ker S_{C^0}$ connecting the identity to $h$. Because $h \in \ker S_{C^0}$ was arbitrary, the subgroup $\ker S_{C^0}$ is path connected.
\end{proof}
	
	\section{Cohomology groups of $\mathbb{G}^{\Omega}(M)$ with coefficients in 
		$\mathcal{C}_0(M,\mathbb{R})$}
	
	Let $\mathcal{C}_0(M,\mathbb{R})$ denote the space of continuous functions on 
	$M$ with zero mean with respect to the normalized volume form $\Omega$:
	\begin{equation}
		\mathcal{C}_0(M,\mathbb{R}) = \left\{ f \in \mathcal{C}(M,\mathbb{R}) \;:\; 
		\int_M f\,\Omega = 0 \right\}.
	\end{equation}
	The right action of $\mathbb{G}^{\Omega}(M)$ given by 
	$(f\cdot h)(x) = f(h(x))$ preserves the volume measure, so 
	$\mathcal{C}_0(M,\mathbb{R})$ is a well-defined 
	$\mathbb{G}^{\Omega}(M)-$module.
	
	\subsection{The intrinsic $1-$cocycle}
	
	For each closed $1$-form $\alpha\in\mathcal{Z}^1(M)$ and each 
	$h\in \mathbb{G}^{\Omega}(M)$, choose a continuous isotopy 
	$H=\{h^t\}_{t\in[0,1]}\in \mathcal{P}_\ast\mathbb{G}^{\Omega}(M)$ from $\mathrm{id}_M$ to $h$ 
	and define 
	\[
	\tau_\alpha(h)
	:= \chi(h,\alpha).
	\]
	
	\begin{proposition}[Cocycle property]\label{prop:cocycle_intrinsic}
		For any closed $1-$form \(\alpha\), the map 
		\(\tau_\alpha : \mathbb{G}^{\Omega}(M) \to \mathcal{C}_0(M,\mathbb{R})\) 
		is a continuous $1-$cocycle, i.e.,
		\begin{equation}
			\tau_\alpha(h_1\circ h_2) 
			= \tau_\alpha(h_2) + \tau_\alpha(h_1)\circ h_2 
			\qquad \forall\,h_1,h_2\in \mathbb{G}^{\Omega}(M).
		\end{equation}
	\end{proposition}
	
	\begin{proof}
		This is an immediate consequence of Corollary \ref{cor3}.
	\end{proof}
	\subsection{Cohomology of $\mathbb{G}^{\Omega}(M)$}
	
	Let $C^k(M,\Omega)$ denote the set of all continuous maps
	\[
	F: \bigl(\mathbb{G}^{\Omega}(M)\bigr)^{\times k} \to \mathcal{C}_0(M,\mathbb{R}),
	\]
	endowed with the compact-open topology.	The coboundary operator $\delta^k: C^k(M,\Omega)\to C^{k+1}(M,\Omega)$ is
	defined by the usual homogeneous formula
	\[
	\begin{aligned}
		\delta^k(F)(h_1,\dots,h_{k+1})
		&= F(h_2,\dots,h_{k+1}) - F(h_1,\dots,h_k)\circ h_{k+1} \\
		&\quad + \sum_{i=1}^k (-1)^i F(h_1,\dots,h_i h_{i+1},\dots,h_{k+1}).
	\end{aligned}
	\]
	One checks that $\delta^{k+1}\circ\delta^k = 0$, yielding the cochain complex
	\[
	\mathcal{C}_0(M,\mathbb{R}) \xrightarrow{\delta^0} C^1(M,\Omega)
	\xrightarrow{\delta^1} \cdots \xrightarrow{\delta^{k-1}} C^k(M,\Omega)
	\xrightarrow{\delta^k} C^{k+1}(M,\Omega) \xrightarrow{\delta^{k+1}} \cdots
	\]
	The $k$-th continuous cohomology group is
	\[
	H^k(\mathbb{G}^{\Omega}(M),\mathcal{C}_0(M,\mathbb{R}))
	\coloneqq \ker\delta^k / \operatorname{im}\delta^{k-1}.
	\]

	\subsection{Zeroth cohomology}
	
	The zeroth cohomology consists of invariants:
	\[
	H^0(\mathbb{G}^{\Omega}(M),\mathcal{C}_0(M,\mathbb{R})) 
	= \{ f\in\mathcal{C}_0(M,\mathbb{R}) : f\circ h = f 
	\;\forall h\in \mathbb{G}^{\Omega}(M) \}.
	\]
	Since $G_\Omega(M)\subset \mathbb{G}^{\Omega}(M)$ acts transitively on $M$ 
	(Boothby~\cite{Boo}), the only invariant continuous functions are constants.  
	Because functions in $\mathcal{C}_0(M,\mathbb{R})$ have mean zero, the only 
	constant is $f\equiv 0$.  Hence 
	$H^0(\mathbb{G}^{\Omega}(M),\mathcal{C}_0(M,\mathbb{R})) = \{0\}$.

	\subsection{First cohomology: the Ismagilov isomorphism}
	
	For each closed $1$-form $\alpha$, the map $\tau_\alpha$ is a 
	$1$-cocycle by Proposition~\ref{prop:cocycle_intrinsic}; thus it defines a 
	cohomology class 
	$[\tau_\alpha] \in H^1(\mathbb{G}^{\Omega}(M),\mathcal{C}_0(M,\mathbb{R}))$.
	
	\begin{theorem}[Isomorphism for continuous cohomology]\label{thm:Ismagilov}
		Let $(M,\Omega)$ be a closed oriented manifold of dimension $n$.  
		The assignment $\alpha \mapsto \tau_\alpha$ induces an 
		isomorphism of real vector spaces
		\begin{equation}
			\iota : H^1(M,\mathbb{R}) \overset{\cong}{\longrightarrow} 
			H^1\bigl(\mathbb{G}^{\Omega}(M), \mathcal{C}_0(M,\mathbb{R})\bigr),
			\qquad [\alpha] \longmapsto [\tau_\alpha].
		\end{equation}
	\end{theorem}
	
	\begin{proof}
		\textbf{Well-definedness on cohomology.}

		For an exact form $\alpha = df$, a direct computation gives
		\begin{equation}
			\tau_{df}(h)(x) = f(x) - f(h(x)) = (\delta^0 f)(h)(x),
		\end{equation}
		which is a coboundary. Hence $[\tau_{df}] = 0$.\\
		
		\noindent\textbf{Injectivity.}  
		Suppose $[\tau_\alpha] = 0$.  Then there exists a continuous 
		$f\in\mathcal{C}_0(M,\mathbb{R})$ such that for all 
		$h\in \mathbb{G}^{\Omega}(M) $,
		\begin{equation}\label{eq:coboundary}
			\tau_\alpha(h) = (\delta^0 f)(h) = f - f\circ h.
		\end{equation}
		
		Restrict this identity to the subgroup $G_\Omega(M)$ of smooth 
		volume-preserving diffeomorphisms. For any smooth divergence-free vector 
		field $X\in\mathfrak{X}_{\text{vol}}(M)$ with flow $\phi^t$, evaluating (\ref{eq:coboundary}) yields $f(x) - f(\phi^t(x)) = \tau_\alpha(\phi^t)(x)$. Because $\phi^t$ is a smooth isotopy, the intrinsic flux and the orbit integral are smooth, making $\tau_\alpha(\phi^t)(x)$ a smooth map in both $t$ and $x$. Consequently, $t \mapsto f(\phi^t(x))$ must be differentiable at $t =0$. Therefore, the classical directional derivative $ (X.f)(x)$ exists everywhere. Differentiating the equality $f(x) - f(\phi^t(x)) = \tau_\alpha(\phi^t)(x)$ at $t=0$ gives
		\begin{equation}\label{eq:coboundary2}
			(X\cdot f)(x) = \alpha(X)(x) - \int_M \alpha(X)\,\Omega.
		\end{equation}
		The right-hand side is a smooth function on $M$ for every $X\in\mathfrak{X}_{\text{vol}}(M)$. By Lemma~\ref{lem:spanning}, we can locally choose a frame $\{X_1, \dots, X_n\}$ of $TM$ consisting entirely of smooth divergence-free vector fields. Equation (\ref{eq:coboundary2}) implies that the directional derivative $X_i \cdot f$ exists and is a smooth function for each $i=1,\dots,n$. Because $f$ is continuous and its first-order partial derivatives along a smooth local frame exist and are continuous, standard multivariable calculus (by integrating along the smooth flows) ensures that $f$ is of class $C^1$. We can then apply $X_j$ to both sides of (\ref{eq:coboundary2}) to deduce that $X_j(X_i \cdot f)$ is smooth, which forces $f$ to be $C^2$. Iterating this bootstrapping argument along the divergence-free frame guarantees that all higher-order derivatives of $f$ exist and are continuous. Thus, $f\in C^\infty_0(M,\mathbb{R})$.	Now, the identity (\ref{eq:coboundary}) holds in the smooth cochain complex 
		of $G_\Omega(M)$ with coefficients in $C^\infty_0(M,\mathbb{R})$.  By 
		Ismagilov's classical theorem~\cite{Ism} for smooth coefficients, the map 
		$[\alpha]\mapsto[\tau_\alpha|_{G_\Omega(M)}]$ is injective; 
		hence $[\alpha]=0$ in $H^1(M,\mathbb{R})$.\\
		\medskip
		\noindent\textbf{Surjectivity.}
		Let \(c: \mathbb{G}^{\Omega}(M) \to \mathcal{C}_0(M,\mathbb{R})\) be a continuous \(1\)-cocycle. We must find a closed \(1\)-form \(\alpha_0\) and \(f\in\mathcal{C}_0(M,\mathbb{R})\) such that
		\[
		c(h) = \tau_{\alpha_0}(h) + f - f\circ h \qquad \forall h\in\mathbb{G}^{\Omega}(M).
		\]
		
		Restrict \(c\) to the smooth subgroup \(G_\Omega(M)\). This gives a continuous \(1\)-cocycle
		\[
		c_0 := c|_{G_\Omega(M)} : G_\Omega(M) \longrightarrow \mathcal{C}_0(M,\mathbb{R}).
		\]
		By Ismagilov's theorem for continuous coefficients (Theorem~\ref{thm:Ismagilov}), the map
		$
		[\alpha] \longmapsto [\tau_\alpha|_{G_\Omega(M)}]
		$
		is an isomorphism from 
		$
		H^1(M;\mathbb{R})$ onto $ H^1_{\mathrm{cont}}\bigl(G_\Omega(M), \mathcal{C}_0(M,\mathbb{R})\bigr).$	Hence, there exist a closed \(1\)-form \(\alpha_0\) and \(f\in\mathcal{C}_0(M,\mathbb{R})\) such that
		\[
		c_0(h) = \tau_{\alpha_0}(h) + f - f\circ h \qquad \forall h\in G_\Omega(M).
		\tag{*}
		\]
		Now the right-hand side of (*) is continuous on \(\mathbb{G}^{\Omega}(M)\): \(\tau_{\alpha_0}\) is continuous by construction, and \(h\mapsto f\circ h\) is continuous in the \(C^0\)-topology because \(f\) is uniformly continuous on the compact manifold \(M\). The left-hand side \(c\) is also continuous by hypothesis. Since \(G_\Omega(M)\) is \(C^0\)-dense in \(\mathbb{G}^{\Omega}(M)\) (by definition of the latter), the equality (*) extends uniquely to all \(h\in\mathbb{G}^{\Omega}(M)\). Thus \(c\) is cohomologous to \(\tau_{\alpha_0}\), proving surjectivity.

	\end{proof}

	\subsection{Examples and applications}
	
	\subsubsection*{The circle $M=S^{1}$}
	
	Let $M=S^{1}$ with the standard length form $\Omega = d\theta$.  Then 
	$\Homeo^\Omega_0(S^{1})$ coincides with the group $\Homeo_{+}(S^{1})$ of 
	orientation-preserving homeomorphisms.  Closed $1-$forms are multiples of 
	$d\theta$.  For $h\in\Homeo_{+}(S^{1})$ and any isotopy $H$ from 
	$\mathrm{id}_{S^1}$ to $h$, the intrinsic flux is the classical rotation number:
	\begin{equation}
		\widetilde{S}_{C^0}(H)([d\theta]) = \rho(h),
	\end{equation}
	while the orbit integral is $\int_{c_x^H} d\theta = \tilde{h}(x) - x$, where 
	$\tilde{h}$ is a lift of $h$ to the universal cover $\mathbb{R}$.  Hence
	\begin{equation}
		\tau_{d\theta}(h)(x) = \rho(h) - (\tilde{h}(x) - x).
	\end{equation}
	Every continuous $1$-cocycle on $\Homeo_{+}(S^{1})$ with values in 
	$\mathcal{C}_0(S^{1},\mathbb{R})$ is cohomologous to a multiple of 
	$\tau_{d\theta}$ (this is a classical result; see e.g., the 
	discussion in Ismagilov~\cite{Ism}).  Consequently,
	\begin{equation}
		\dim H^{1}\bigl(\mathbb{G}^{\Omega}(S^{1}), \mathcal{C}_0(S^{1},\mathbb{R})\bigr) 
		= b_{1}(S^{1}) = 1.
	\end{equation}
	\subsubsection*{The $2$-torus $\mathbb{T}^2$}
	
	Let $M = \mathbb{T}^2 = \mathbb{R}^2/\mathbb{Z}^2$ with the standard area form 
	$\Omega = dx\wedge dy$.  The first cohomology $H^1(\mathbb{T}^2,\mathbb{R})$ 
	is generated by $[dx]$ and $[dy]$.  By Theorem~\ref{thm:Ismagilov}, the 
	cocycles $\tau_{dx}$ and $\tau_{dy}$ generate 
	$H^1(\mathbb{G}^{\Omega}(\mathbb{T}^2), \mathcal{C}_0(\mathbb{T}^2,\mathbb{R}))$, 
	which is therefore $2$-dimensional.\\
	
	\textbf{Concrete non-trivial cocycle.}
	Let $\theta:\mathbb{R}\to\mathbb{R}$ be a $1$-periodic continuous function, 
	everywhere non-differentiable (e.g., a Weierstrass function).  Define the shear 
	homeomorphism
	\begin{equation}
		h(x,y) = (x + \theta(y),\, y) \pmod 1,
	\end{equation}
	which is isotopic to the identity via $h^t(x,y) = (x + t\theta(y), y)$.  
	Each $h^t$ preserves $\Omega$, so $h\in \mathbb{G}^{\Omega}(\mathbb{T}^2)$.  
	For $\alpha = dx$, the orbit integral is $
	\int_{c_{(x,y)}^H} dx = \theta(y).$ 
	The intrinsic flux is the spatial average
	\begin{equation}
		\widetilde{S}_{C^0}(H)([dx]) 
		= \int_{\mathbb{T}^2} \theta(y)\,dx\,dy 
		= \int_0^1 \theta(y)\,dy \eqqcolon \bar\theta.
	\end{equation}
	Hence,
	\begin{equation}
		\tau_{dx}(h)(x,y) = \bar\theta - \theta(y).
	\end{equation}
	This class is non-zero: if it were a coboundary, we would have 
	$\bar\theta - \theta(y) = f(x,y) - f(x+\theta(y),y)$ for some continuous $f$.  
	Choosing $y_0$ with $\theta(y_0)\neq\bar\theta$ and restricting to the 
	invariant circle $y=y_0$ (where $h$ is a pure translation by $\theta(y_0)$) 
	yields a contradiction by iterating the coboundary equation.  Thus 
	$[\tau_{dx}]\neq 0$, and similarly $[\tau_{dy}]\neq 0$.  
	By the theorem, these two classes form a basis of the cohomology group.\\
	
	\textbf{Relation to fixed-point theory.}
	On $\mathbb{T}^2$, the flux homomorphism takes values in 
	$H^1(\mathbb{T}^2,\mathbb{R})/\Gamma_\Omega \cong \mathbb{T}^2$.  An element 
	$h\in \mathbb{G}^{\Omega}(\mathbb{T}^2)$ has zero flux if and only if its cocycle 
	$\tau_\alpha(h)$ is a coboundary for every closed $1$-form 
	$\alpha$.  By Franks' theorem~\cite{Franks1988}, such homeomorphisms possess 
	at least two fixed points.  Thus the cohomological vanishing 
	$[\tau_\alpha]=0$ for all $\alpha$ has a concrete dynamical 
	consequence: the existence of fixed points.
	
	\subsubsection*{The complex projective plane $\mathbb{CP}^2$}
	
	Let $M = \mathbb{CP}^2$ equipped with the volume form 
	$\Omega = \frac{1}{2}\omega^2$, where $\omega$ is the Fubini-Study symplectic 
	form. Here, we replace  $\mathbb{G}^{\Omega}(\mathbb{CP}^2) $ by the $C^0-$closure of $G_\Omega(\mathbb{CP}^2) $. Since $\mathbb{CP}^2$ is simply connected, 
	$H^1(\mathbb{CP}^2,\mathbb{R}) = 0$.  Theorem~\ref{thm:Ismagilov} gives
	\[
	H^1\bigl(\mathbb{G}^{\Omega}(\mathbb{CP}^2), \mathcal{C}_0(\mathbb{CP}^2,\mathbb{R})\bigr) = 0.
	\]
	Thus every continuous $1$-cocycle on $\mathbb{G}^{\Omega}(\mathbb{CP}^2)$ with 
	values in $\mathcal{C}_0(\mathbb{CP}^2,\mathbb{R})$ is a coboundary: there 
	exists a continuous function $f$ (with zero mean) such that 
	$c(h) = f - f\circ h$ for all $h\in \mathbb{G}^{\Omega}(\mathbb{CP}^2)$.  In 
	contrast, the de Rham cohomology group $H^2(\mathbb{CP}^2,\mathbb{R}) \cong \mathbb{R}$ is non-trivial, 
	so the second cohomology of the homeomorphism group (discussed below) can 
	carry interesting invariants.
	
	\subsection{The second cohomology and central extensions}
	
	The intrinsic flux homomorphism 
	$\widetilde{S}_{C^0}:\mathcal{P}_\ast\mathbb{G}^{\Omega}(M)  \to 
	H^{n-1}(M,\mathbb{R})$ induces a homomorphism $ S_{C^0}$ on the group itself by 
	evaluating on any isotopy from the identity.  More precisely, for 
	$h\in \mathbb{G}^{\Omega}(M)$, choose any continuous isotopy $H$ from $\mathrm{id}_M$ 
	to $h$ and set
	\begin{equation}
		S_{C^0}(h) \coloneqq \widetilde{S}_{C^0}(H) \pmod{\Gamma_\Omega}
		\;\in\; H^{n-1}(M,\mathbb{R})/\Gamma_\Omega.
	\end{equation}
	This is independent of the choice 
	of $H$ and defines a continuous group homomorphism
	\begin{equation}
		S_{C^0} : \mathbb{G}^{\Omega}(M) \longrightarrow 
		H^{n-1}(M,\mathbb{R})/\Gamma_\Omega.
	\end{equation}
	Recall that $\Gamma_\Omega$ is a discrete subgroup of 
	$H^{n-1}(M,\mathbb{R})$ (Theorem~\ref{thm:C0rigidity}).  The exact sequence
	\begin{equation}
		0 \longrightarrow \Gamma_\Omega \longrightarrow H^{n-1}(M,\mathbb{R}) 
		\longrightarrow H^{n-1}(M,\mathbb{R})/\Gamma_\Omega \longrightarrow 0,
	\end{equation}
	is a central extension of abelian groups.  Pulling back via 
	$S_{C^0} $ yields a central extension of topological groups:
	\begin{equation}
		0 \longrightarrow \Gamma_\Omega \longrightarrow 
		\widehat{\mathbb{G}^{\Omega}(M)} \longrightarrow 
		\mathbb{G}^{\Omega}(M) \longrightarrow 0,
	\end{equation}
	where
	\[
	\widehat{\mathbb{G}^{\Omega}(M)} = \bigl\{ (h, v) \in \mathbb{G}^{\Omega}(M) \times 
	H^{n-1}(M,\mathbb{R}) \;\big|\; 
	S_{C^0} (h) = \pi(v) \bigr\},
	\]
	with group law $(h_1,v_1)(h_2,v_2) = (h_1h_2, v_1+v_2)$ and $\pi$ the 
	quotient map.  This central extension defines a cohomology class
	\begin{equation}
		\mathfrak{e}_\Omega \in H^2\bigl(\mathbb{G}^{\Omega}(M), \Gamma_\Omega\bigr).
	\end{equation}
	\begin{theorem}[Second cohomology and the flux central extension]\label{thm:H2_flux}
		The class $[\mathfrak{e}_\Omega] \in 
		H^2(\mathbb{G}^{\Omega}(M), \Gamma_\Omega)$ is non-trivial whenever 
		$\Gamma_\Omega \neq 0$.  Moreover, for any non-trivial homomorphism 
		$ D : \Gamma_\Omega \to \mathbb{R}$, the push-forward 
		$D_*[\mathfrak{e}_\Omega]$ is a non-zero element of 
		$H^2(\mathbb{G}^{\Omega}(M), \mathbb{R})$.
	\end{theorem}
	
	\begin{proof}
		Let $V := H^{n-1}(M;\mathbb{R})$. Since $\Gamma_\Omega$ is a discrete subgroup 
		of the vector space $V$ (Theorem~\ref{thm:C0rigidity}), the quotient 
		$V/\Gamma_\Omega$ is a $K(\Gamma_\Omega, 1)$-space, so its fundamental group 
		is canonically isomorphic to $\Gamma_\Omega$. Suppose the extension splits.  Then there exists a continuous homomorphism 
		$\tilde{L}: \mathbb{G}^{\Omega}(M) \to V$ such that 
		$\pi \circ \tilde{L} = S_{C^0}$, where $\pi: V \to V/\Gamma_\Omega$ is the 
		quotient map.  Because $\Gamma_\Omega \neq 0$, pick a loop $\gamma$ in 
		$ \mathbb{G}^{\Omega}(M)$ based at $\mathrm{id}_M$ whose flux 
		$v := \widetilde{S}_{C^0}(\gamma) \in \Gamma_\Omega$ is non-zero (such a loop 
		exists by Theorem~\ref{thm:C0rigidity}).  Since $\tilde{L}$ is a 
		homomorphism, $\tilde{L}\circ\gamma$ is a loop in the vector space $V$, 
		hence contractible.  Its projection 
		$\pi \circ \tilde{L}\circ\gamma = S_{C^0} \circ\gamma$ 
		is therefore null-homotopic in the quotient 
		$V/\Gamma_\Omega$.  But $S_{C^0} \circ\gamma$ is precisely the loop 
		representing the non-zero element $v \in \Gamma_\Omega \cong 
		\pi_1(V/\Gamma_\Omega)$; such a loop is not null-homotopic.  
		This contradiction shows that no splitting exists, so $[\mathfrak{e}_\Omega] \neq 0$. 	For the second statement, let $D: \Gamma_\Omega \to \mathbb{R}$ be any 
		non-trivial homomorphism. If $D_*[\mathfrak{e}_\Omega]=0$ in 
		$H^2(\mathbb{G}^{\Omega}(M), \mathbb{R})$, then the pushed-forward extension 
		splits, yielding a continuous homomorphism 
		$\tilde{L}_D: \mathbb{G}^{\Omega}(M) \to \mathbb{R}$ that lifts 
		$D \circ S_{C^0}$, i.e., $\pi_{\mathbb{R}} \circ \tilde{L}_D = D \circ S_{C^0}$.
		Now consider any loop $\gamma$ in $\mathbb{G}^{\Omega}(M)$ with flux 
		$v = \widetilde{S}_{C^0}(\gamma) \in \Gamma_\Omega$. Since $\tilde{L}_D$ is a 
		homomorphism and $\gamma$ is a loop, the path $\tilde{L}_D \circ \gamma$ is 
		a loop in $\mathbb{R}$ based at $0$. But $\mathbb{R}$ is contractible, so 
		every loop in $\mathbb{R}$ is null-homotopic. Hence its projection 
		$D(S_{C^0}(\gamma(t)))$ represents the zero element of 
		$\pi_1(\mathbb{R}/D(\Gamma_\Omega))$, which is $D(\Gamma_\Omega)$. 
		Therefore $D(v) = 0$. Since $\gamma$ was arbitrary and the fluxes of loops 
		generate $\Gamma_\Omega$ (Theorem~\ref{thm:C0rigidity}), we obtain 
		$D(\Gamma_\Omega) = 0$, contradicting the non-triviality of $D$.  
		Hence $D_*[\mathfrak{e}_\Omega] \neq 0$.
	\end{proof}
	
	\section{Applications to Surfaces and Open Problems}
	\label{sec:applications}
	
	In this section we collect the implications of our results for closed oriented surfaces, where the area form is simultaneously a volume form and a symplectic form. We also provide concrete topological criteria regarding fixed points and non-displaceability on the torus.
	
	\subsection{Fixed-point theory on $\mathbb{T}^2$}
	
	For the torus, the flux homomorphism becomes essentially the rotation vector 
	(or average displacement), where we identify \(H^1(\mathbb{T}^2,\mathbb{R})\) 
	with \(\mathbb{R}^2\) via the basis \([dx], [dy]\). 
	Theorem~\ref{thm:C0rigidity} shows that the flux group 
	\(\widetilde{\Gamma}_\Omega = \Gamma_\Omega\) is a full lattice of rank \(2\) 
	in \(H^1(\mathbb{T}^2,\mathbb{R}) \cong \mathbb{R}^2\). 
	Therefore, the quotient \(H^1(\mathbb{T}^2,\mathbb{R}) / \Gamma_\Omega\) 
	is itself a torus \(\mathbb{T}^2\), and the topological flux 
	\(S_{C^0}:\mathbb{G}^{\Omega}(\mathbb{T}^2) \to \mathbb{T}^2\) 
	is a well-defined, continuous homomorphism. 
	In particular, the kernel \(\ker S_{C^0}\) consists exactly of those 
	area-preserving homeomorphisms that are \(C^0\)-limits of Hamiltonian 
	diffeomorphisms. Now, a classical theorem of Franks~\cite{Franks1988} states that every 
	area-preserving homeomorphism of \(\mathbb{T}^2\) that is isotopic 
	to the identity and has zero rotation vector possesses at least two 
	distinct fixed points. Because our construction identifies the zero-flux 
	condition \(S_{C^0}(h) = 0\) with the hypothesis of Franks' theorem 
	(via the identification of flux with rotation vector on the torus), 
	and because $G_\Omega(\mathbb{T}^2)$ is $C^0-$dense in $\mathbb{G}^{\Omega}(\mathbb{T}^2)$, every element of \(\ker S_{C^0}\) 
	automatically satisfies Franks' hypothesis. Hence, our result implies 
	the following topological criterion:
	
	\begin{corollary}[Fixed points on \(\mathbb{T}^2\)]
		Let \(h \in \mathbb{G}^{\Omega}(\mathbb{T}^2)\) be an area-preserving 
		homeomorphism isotopic to the identity. If the topological flux 
		\(S_{C^0}(h) = 0\) (in particular, if \(h\) is a \(C^0\)-limit of 
		Hamiltonian diffeomorphisms), then \(h\) has at least two distinct 
		fixed points.
	\end{corollary}
	
	Conversely, if \(S_{C^0}(h) \neq 0\), the homeomorphism $h$ can be 
	fixed-point-free. For example, an irrational translation 
	\((x,y) \mapsto (x+\alpha, y+\beta)\) with \((\alpha, \beta) \notin 
	\mathbb{Z}^2\) has non-zero flux and no fixed points. Thus, our work 
	provides a concrete topological and cohomological criterion that separates 
	maps with forced fixed points (zero flux) from those that may be 
	fixed-point-free (non-zero flux).
	
	\subsection{Non-displaceable Lagrangian submanifolds on $\mathbb{T}^2$}
	
	Let $(\mathbb{T}^2,\Omega = dx\wedge dy)$ be the standard symplectic $2$-torus with the normalized area form. A Lagrangian submanifold of $\mathbb{T}^2$ is simply a simple closed curve. We say that $h\in \mathbb{G}^{\Omega}(\mathbb{T}^2)$ \emph{displaces} a Lagrangian $L$ if $h(L)\cap L = \emptyset$.

	\begin{proposition}[Non-displaceability by zero-flux homeomorphisms]
		\label{prop:non_disp}
		Let $L\subset\mathbb{T}^2$ be a simple closed curve whose homology class
		$[L]\in H_1(\mathbb{T}^2;\mathbb{Z})$ is non-zero. Then there exists no
		homeomorphism $h\in\ker S_{C^0}$ such that $h(L)\cap L=\emptyset$.
	\end{proposition}
	
	\begin{proof}
		Assume, for contradiction, that such an $h$ exists. Choose a continuous
		volume-preserving isotopy
		$H_0=\{h_0^t\}_{t\in[0,1]}$ from $\mathrm{id}_{\mathbb{T}^2}$ to $h$.
		Since $h\in\ker S_{C^0}$, Proposition~\ref{P4.4} gives
		$
		\widetilde{S}_{C^0}(H_0)\in\widetilde{\Gamma}_\Omega=\Gamma_\Omega.
		$
		Because every element of $\Gamma_\Omega$ is the flux of some smooth loop of
		volume-preserving diffeomorphisms, there exists a smooth loop
		$\{g_t\}_{t\in[0,1]}$ in $G_\Omega(\mathbb{T}^2)$ with
		$g_0=g_1=\mathrm{id}_{\mathbb{T}^2}$ and
		$
		\widetilde{S}_\Omega(g)=-\widetilde{S}_{C^0}(H_0).
		$	Concatenate $H_0$ with $g$ to obtain a new continuous isotopy
		$H=\{h^t\}_{t\in[0,1]}$ from $\mathrm{id}_{\mathbb{T}^2}$ to $h$ by
		\[
		h^t=
		\begin{cases}
			h_0^{2t}, & 0\le t\le \tfrac12,\\[1mm]
			g_{2t-1}\circ h, & \tfrac12\le t\le 1.
		\end{cases}
		\]
		Then, we have $
		\widetilde{S}_{C^0}(H)
		=
		\widetilde{S}_{C^0}(H_0)+\widetilde{S}_\Omega(g)
		=
		0$
		in $H^1(\mathbb{T}^2;\mathbb{R}).$ 
		Now, consider the $2$-chain\\
		$
		C:=H([0,1]\times L)\subset\mathbb{T}^2,
		$
		parameterized by $(t,p)\mapsto h^t(p)$. Its boundary is, $
		\partial C=h(L)-L.$ 
		Because $h\in\mathbb{G}^{\Omega} (\mathbb{T}^2)$ is isotopic to the identity,
		$h(L)$ is homologous to $L$. Since $h(L)\cap L=\emptyset$, the two disjoint
		homologous simple closed curves bound an annulus of positive area. Hence, 
		$
		\int_C\omega>0,$ 
		where $\omega$ is the standard area form on $\mathbb{T}^2$.	On the other hand, since
		$\widetilde{S}_{C^0}(H)=0$ and $H$ is the $C^0$-limit of smooth
		volume-preserving isotopies $\Phi_i$, for each $i$ the classical
		flux-area relation gives
		$
		\int_{C_i}\omega
		=
		\langle \widetilde{S}_\Omega(\Phi_i),[L]\rangle,
		$ with $
		C_i:=\Phi_i([0,1]\times L).$ 
		Passing to the limit as $i\to\infty$, the left-hand side converges to
		$\int_C\omega$, while the right-hand side converges to
		$
		\langle \widetilde{S}_{C^0}(H),[L]\rangle
		=
		0.
		$
		Therefore, 
		$
		\int_C\omega=0,
		$
		contradicting the positivity obtained above.	Hence no such $h$ exists.
	\end{proof}
	
	\subsection{Strict inclusion of Hamiltonian homeomorphisms}
	\begin{corollary}
		\label{lem:strict_inclusion}
		Let $(M, \omega)$ be a closed symplectic manifold of dimension $2n$ with non-trivial $H^{2n -1}(M; \mathbb{R})$. Then, the inclusion of the group of Hamiltonian homeomorphisms into the identity component of the group of symplectic homeomorphisms is strict. That is, 
		$
		\mathrm{Hameo}(M, \omega) \subsetneq \mathrm{Sympeo}_0(M, \omega).
		$
	\end{corollary}
	
	\begin{proof}
		The volume form $\Omega$ here is induced by the symplectic form $\omega$. As established in the $C^0$-transport framework, the homomorphism $S_\Omega$ extends continuously to the $C^0$-closure of the group, defining the topological flux homomorphism:
		\[
		S_{C^0} : \mathrm{Sympeo}_0(M, \omega) \longrightarrow H^{2n -1}(M; \mathbb{R}) / \Gamma_\Omega.
		\]
		By definition, $\mathrm{Hameo}(M, \omega)$ is the closure of $\mathrm{Ham}(M, \omega)$ under the Hamiltonian topology \cite{OhMuller2007}. Because the topological flux $S_{C^0}$ is continuous with respect to the $C^0$-topology (and hence continuous under the stronger Hamiltonian topology), and it vanishes on $\mathrm{Ham}(M, \omega)$, it must vanish identically on $\mathrm{Hameo}(M, \omega)$. Thus, 
		$
		\mathrm{Hameo}(M, \omega) \subset \ker(S_{C^0}).
		$
		Since $H^{2n-1}(M; \mathbb{R}) \neq 0$, the quotient $H^{2n-1}(M; \mathbb{R}) / \Gamma_\Omega$ is non-trivial. By the surjectivity of the smooth flux, there exists a smooth symplectic diffeomorphism $g \in \mathrm{Symp}_0(M, \omega)$ such that $S_{C^0}(g) = \mathrm{S}_\Omega(g) \neq 0$. 
		Because $g$ is smooth, $g \in \mathrm{Sympeo}_0(M, \omega)$. However, since $S_{C^0}(g) \neq 0$, $g \notin \ker(S_{C^0})$, which implies $g \notin \mathrm{Hameo}(M, \omega)$. Therefore, the inclusion is strict.
	\end{proof}

\subsection{Dimension 3: Non-displaceability of Incompressible Surfaces}

On a \(3\)-manifold, the natural geometric analogue of a Lagrangian submanifold
(or a simple closed curve on a surface) is an incompressible surface. Using
Poincar\'e duality, the topological flux–volume relation can be expressed directly
in terms of the flux functional on \(H^1\).

\begin{proposition}[Non-displaceability of Incompressible Surfaces]
	Let \(M\) be a closed oriented \(3\)-manifold with a volume form \(\Omega\),
	and let \(S \subset M\) be a closed, oriented surface with non-zero homology
	class \([S]\neq 0\) in \(H_2(M;\mathbb Z)\). Let
	$
	\eta_S := \operatorname{PD}([S]) \in H^1(M;\mathbb Z)
	$
	be its Poincar\'e dual. Then no zero-flux homeomorphism
	\(h \in \ker S_{C^0}\) can displace \(S\). That is,
	$
	h(S) \cap S \neq \emptyset$ 	for all $h \in \ker S_{C^0}.
	$
\end{proposition}

\begin{proof}
	Let \(h \in \ker S_{C^0}\). Choose any continuous isotopy
	\(H_0 \in \mathcal{P}_*\mathbb{G}^{\Omega}(M)\) from \(\mathrm{id}\) to \(h\).
	Let
	$
	v := \widetilde{S}_{C^0}(H_0) \in H^2(M;\mathbb{R}).
	$
	Since \(h \in \ker S_{C^0}\), we have \(v \in \Gamma_\Omega\). By definition
	of \(\Gamma_\Omega\), there exists a loop
	\(\gamma \in \pi_1^{\mathrm{path}}(\mathbb{G}^{\Omega}(M))\) based at
	\(\mathrm{id}\) such that
	$
	\widetilde{S}_{C^0}(\gamma) = -v.
	$
	Concatenate \(\gamma\) with \(H_0\) to obtain an isotopy
	\(H \in \mathcal{P}_*\mathbb{G}^{\Omega}(M)\) from \(\mathrm{id}\) to \(h\)
	with
	$
	\widetilde{S}_{C^0}(H) = 0.
	$
	Now apply Lemma~\ref{L-03} to \(H\). We obtain an isotopy
	\(G = \{g^t\}_{t\in[0,1]} \in \mathcal{P}_*\mathbb{G}^{\Omega}(M)\) with
	\(g^0 = \mathrm{id}\), \(g^1 = h\), and
	$
	g^t \in \ker S_{C^0} \qquad \forall t \in [0,1].
	$
	In particular, \(\widetilde{S}_{C^0}(G) = 0\).	Assume for contradiction that \(h(S) \cap S = \emptyset\). Since \(G\) is an
	isotopy from \(\mathrm{id}\) to \(h\), the surface \(S\) sweeps out a
	\(3\)-chain \(C_G\) whose boundary is \(h(S) - S\). Because \(h(S)\) and
	\(S\) are disjoint and homologous, they cobound a region of positive volume:
	$
	\int_{C_G} \Omega = \operatorname{vol}(C) > 0.
	$
	On the other hand, by Poincar\'e duality,
	\[
	\langle \widetilde{S}_{C^0}(G), [S]\rangle
	=
	\widetilde{S}_{C^0}(G)([\eta_S])
	=
	\int_M \mathcal R(G,\eta_S)\,\Omega .
	\]
	For smooth volume-preserving isotopies, the flux–volume relation gives
	$
	\langle \widetilde{S}_\Omega(\Phi), [S]\rangle
	=
	\int_{\Phi([0,1]\times S)} \Omega.$ 
	Passing to the \(C^0\)-limit along smooth approximations of \(G\), we obtain
	$
	\langle \widetilde{S}_{C^0}(G), [S]\rangle
	=
	\int_{C_G} \Omega.$
	But \(\widetilde{S}_{C^0}(G) = 0\), so the right-hand side must be zero.
	This contradicts the positivity of the volume. Hence no such displacing
	\(h\) exists.
\end{proof}

	\subsubsection*{Persistent homology and Floer barcodes}
	
	In recent years, the application of persistent homology to filtered Hamiltonian
	Floer theory has revolutionized the study of $C^0$-symplectic dynamics.
	Pioneered by Polterovich, Shelukhin \cite{PS}, and Usher \cite{Usher}, the Floer
	action functional generates a persistence module whose ``barcodes'' (intervals
	representing the birth and death of Floer homology classes) are remarkably
	stable under the $C^0$-topology. Because our framework brings the continuous
	topological flux into the $C^0$-regime, we can formulate a concrete algebraic
	statement connecting the topological flux $S_{C^0}$ (and the cocycle
	$\tau_\alpha$) to the translational shifts of Floer barcodes in dimension 2.
	Here, on a closed surface $\Sigma_g$, the area form is simultaneously a volume
	form and a symplectic form, so the volume-preserving homeomorphisms coincide
	with the symplectic homeomorphisms, and the flux group lies in
	$H^1(\Sigma_g;\mathbb R)$.	For an area-preserving diffeomorphism of $\Sigma_g$ with non-zero flux, the
	classical symplectic action functional is multi-valued, and Floer homology must
	be defined over a Novikov ring. Geometrically, this manifests as a
	\emph{periodic barcode}: the set of bars is invariant under translation by the
	periods of the flux, i.e., by elements of $\Gamma_\omega$. The translational
	symmetry of the barcode is exactly the flux group.	Because Theorem~\ref{thm:C0rigidity} proves that the flux group is
	$C^0$-rigid (on $\mathbb{G}^{\Omega}(M)$ and, by Fathi's theorem, on the full
	identity component), the algebraic structure required for Floer--Novikov
	homology is stable under $C^0$-limits. In particular, Floer barcodes are finite
	and well-defined without Novikov coefficients precisely when the flux is zero,
	which in the symplectic (Hamiltonian) category corresponds to
	$\tau_\alpha \equiv 0$ for all closed $1$-forms $\alpha$. Thus the topological
	flux acts as a gatekeeper: when $[\tau_\alpha] \neq 0$, the non-trivial
	macroscopic displacement forces periodic barcodes; when $[\tau_\alpha] = 0$,
	the Hamiltonian dynamics allow finite bottleneck distances and robust
	$C^0$-metrics.
	
\begin{theorem}[$C^0$-flux conjecture for $\mathbb{G}^{\Omega}(M)$]
	\label{thm:C0_flux_conjecture}
	Let $(M,\Omega)$ be a closed oriented manifold of dimension $n\ge 2$.  The
	topological flux group
	\[
	\widetilde{\Gamma}_\Omega
	= \widetilde{S}_{C^0}\bigl(\pi_1^{\mathrm{path}}(\mathbb{G}^{\Omega}(M))\bigr)
	\]
	is discrete in $H^{n-1}(M;\mathbb R)$, and $\ker S_{C^0}$ is $C^0$-closed in
	$\mathbb{G}^{\Omega}(M)$.  If, in addition, $(M,\Omega)$ satisfies the local
	$C^0$-generation property, then the same conclusions hold with
	$\mathbb{G}^{\Omega}(M)$ replaced by $\operatorname{Homeo}_0^\Omega(M)$.
\end{theorem}

\begin{proof}
	By Theorem~\ref{thm:C0rigidity}, $\widetilde{\Gamma}_\Omega = \Gamma_\Omega$,
	which is a discrete lattice of rank $b_1(M)$ by Banyaga's theorem.  Hence
	$\widetilde{\Gamma}_\Omega$ is discrete. 	The equivalence with $\ker S_{C^0}$ being $C^0-$closed follows from 
	standard topological group theory: since $S_{C^0}$ is continuous,  $ \ker S_{C^0} = S_{C^0}^{-1}(\{0\}),$ 
	is closed in $ \mathbb{G}^{\Omega}(M)$ if and only if $\{0\}$ is closed in 
	the quotient $H^{n-1}(M;\mathbb{R})/\widetilde{\Gamma}_\Omega$.  
	The point set $\{0\}$ is closed in the quotient because 
	$\widetilde{\Gamma}_\Omega$ is discrete. Thus, $\ker S_{C^0}$ is $C^0-$closed. 	For the extension to $\operatorname{Homeo}_0^\Omega(M)$, apply
	Theorem~\ref{thm:extension}.
\end{proof}

	The $C^0-$flux conjecture, a central motivating question in topological 
	dynamics, posits that the topological flux group is discrete, or 
	equivalently, that the subgroup of homeomorphisms with vanishing flux 
	is $C^0$-closed in the identity component of the homeomorphism group. Particularly striking is the implication for dimension two. On a closed 
	oriented surface $\Sigma_g$, an area form is simultaneously a volume 
	form and a symplectic form, and the volume flux (in $H^1$) coincides 
	with the symplectic flux. Consequently, our framework provides a complete, classical resolution of the symplectic $C^0-$flux conjecture for all closed surfaces of non-trivial genus, bypassing the need for Floer-homological machinery required in higher-dimensional symplectic topology.

	\subsection{Extension to the full identity component}
	
	The main results of this paper are stated for the path-closure group
	\(\mathbb{G}^{\Omega}(M)\). The reason is that the topological flux and the
	transport cocycle are defined using limits of smooth volume-preserving
	isotopies, and it is not known in general whether every continuous
	volume-preserving isotopy can be approximated uniformly by smooth ones.
	In this subsection we give a sufficient condition under which all results
	extend to the full identity component \(\operatorname{Homeo}_0^\Omega(M)\).
	
\begin{definition}[Local \(C^0\)-generation property]
	We say that a closed oriented manifold \((M,\Omega)\) satisfies the
	\emph{local \(C^0\)-generation property} if for every \(\varepsilon>0\)
	there exists \(\delta>0\) such that, whenever
	\(h\in\operatorname{Diff}^\infty_\Omega(M)\) satisfies
	\(d_0(h,\mathrm{id}_M)<\delta\), there exists a smooth volume-preserving
	isotopy \(H=\{h^t\}_{t\in[0,1]}\) from \(h^0=\mathrm{id}_M\) to \(h^1=h\)
	with
	\[
	\max_{t\in[0,1]} d_0(h^t,\mathrm{id}_M) < \varepsilon.
	\]
	In other words, any \(C^0\)-small volume-preserving diffeomorphism can be
	joined to the identity by a \(C^0\)-small smooth volume-preserving isotopy.
\end{definition}

\begin{remark}[Examples where the property holds]
	The local \(C^0\)-generation property is known to hold in the following
	cases:
	\begin{enumerate}
		\item On the circle \(M=S^1\) with the standard length form, every
		orientation-preserving volume-preserving diffeomorphism isotopic to
		the identity is a rotation, and the rotation isotopy has exactly the
		required \(C^0\)-size.
		\item On a torus \(\mathbb T^n\), translations \(x\mapsto x+a\)
		are volume-preserving; the linear isotopy \(x\mapsto x+ta\) realizes
		the property for translations. However, this does not cover all volume-preserving diffeomorphisms
		of \(\mathbb T^n\).
		\item More generally, on any compact connected Lie group with a
		bi-invariant volume form, left translations by elements close to the
		identity admit small one-parameter subgroups of left translations.
		\item For any smooth divergence-free vector field \(X\) with small
		\(C^0\)-norm, the flow \(\phi_X^t\) is a smooth volume-preserving
		isotopy whose \(C^0\)-diameter is controlled by \(\|X\|_{C^0}\).
		Thus diffeomorphisms that are time-one maps of small divergence-free
		flows satisfy the condition.
	\end{enumerate}
	However, for a general closed manifold it is not known whether every
	\(C^0\)-small volume-preserving diffeomorphism can be generated by a
	\(C^0\)-small smooth isotopy. This is essentially the obstacle that
	prevents the unconditional extension of the main theorem to
	\(\operatorname{Homeo}_0^\Omega(M)\). The formulation in terms of the
	path-closure group \(\mathbb{G}^{\Omega}(M)\) is the appropriate
	general setting.
\end{remark}

\begin{theorem}[Conditional extension to \(\operatorname{Homeo}_0^\Omega(M)\)]
	\label{thm:extension}
	Let \((M,\Omega)\) be a closed oriented manifold satisfying the local
	\(C^0\)-generation property. Then smooth volume-preserving isotopies are
	dense in \(\mathcal P\operatorname{Homeo}_0^\Omega(M)\) with respect to the
	\(C^0\)-metric. Consequently,
	$
	\mathbb{G}^{\Omega}(M)=\operatorname{Homeo}_0^\Omega(M),
	$
	and the isomorphism
	\[
	H^1_{\mathrm{cont}}\bigl(\mathbb{G}^{\Omega}(M),C^0_0(M,\mathbb R)\bigr)
	\cong H^1(M;\mathbb R),
	\]
	from Theorem~\ref{thm:Ismagilov} holds with
	\(\mathbb{G}^{\Omega}(M)\) replaced by
	\(\operatorname{Homeo}_0^\Omega(M)\).
\end{theorem}

\begin{proof}
	We shall use the following standard fact.
	
\begin{lemma}[Müller-Sikorav approximation]\label{Lem: Apro}
	Let \((M,\Omega)\) be a closed oriented manifold of dimension \(n\neq 4\).
	If \(h\in\operatorname{Homeo}_0^\Omega(M)\), then for every \(\eta>0\)
	there exists a smooth volume-preserving diffeomorphism
	\(\psi\in\operatorname{Diff}^\infty_\Omega(M)\) such that
	$
	d_{C^0}(h,\psi)<\eta.
	$
	For \(n=4\), the corresponding statement is open. In that case one may
	either assume it as an additional hypothesis or work directly with
	\(\mathbb{G}^{\Omega}(M)\) instead of \(\operatorname{Homeo}_0^\Omega(M)\).
\end{lemma}

\begin{proof}
	Since \(h\in\operatorname{Homeo}_0^\Omega(M)\), the homeomorphism \(h\) is
	isotopic to the identity, hence isotopic to a diffeomorphism.
	
	\begin{enumerate}
		\item For \(n\le 3\), Munkres' smoothing theorem
		\cite{munkres-obstructions,munkres-higher} gives a smooth
		diffeomorphism \(f\in\operatorname{Diff}^\infty(M)\) with
		\(d_{C^0}(h,f)<\eta/2\).
		
		\item For \(n\ge 5\), Müller's theorem \cite[Theorem 1]{Mull-1}
		gives a smooth diffeomorphism \(f\in\operatorname{Diff}^\infty(M)\)
		with \(d_{C^0}(h,f)<\eta/2\), because \(h\) is isotopic to a
		diffeomorphism.
	\end{enumerate}
	
	In both cases we obtain a smooth diffeomorphism \(f\) uniformly close to
	\(h\). Since \(h\) preserves \(\Omega\), the pull-back \(f^*\Omega\) is a
	volume form close to \(\Omega\). By Sikorav's theorem \cite{Sik} (see also
	\cite{Y.Oh}), a volume-preserving homeomorphism that is approximable by
	diffeomorphisms is approximable by volume-preserving diffeomorphisms.
	Hence there exists
	\(\psi\in\operatorname{Diff}^\infty_\Omega(M)\) with
	$
	d_{C^0}(f,\psi)<\eta/2.
	$
	The triangle inequality gives
	\[
	d_{C^0}(h,\psi)
	\le d_{C^0}(h,f)+d_{C^0}(f,\psi)
	< \eta .
	\]
	This proves the lemma.
\end{proof}
	
	Now let \(H=\{h^t\}_{t\in[0,1]}\in\mathcal P\operatorname{Homeo}_0^\Omega(M)\)
	and fix \(\varepsilon>0\). Choose \(N\in\mathbb N\) and set
	\(\eta=\varepsilon/(4N)\). By the local \(C^0\)-generation property, there
	is \(\delta>0\) such that any smooth volume-preserving diffeomorphism within
	\(\delta\) of the identity can be joined to the identity by a smooth
	volume-preserving isotopy of \(C^0\)-size \(<\eta\). Shrinking \(\eta\) if
	necessary, we may assume \(2\eta<\delta\).	Choose a partition
	$
	0=t_0<t_1<\cdots<t_N=1,
	$
	such that, for all \(i=1,\dots,N\) and all \(s\in[0,1]\),
	$
	d_{C^0}(g_{i,s},\mathrm{id})<\eta,
	$
	where
	$
	g_{i,s}:=(h^{t_{i-1}})^{-1}\circ h^{t_{i-1}+s(t_i-t_{i-1})}.
	$
	For each \(i\), set \(g_i:=g_{i,1}\). By Lemma \ref{Lem: Apro}, there is a smooth
	volume-preserving diffeomorphism
	\(\psi_i\in\operatorname{Diff}^\infty_\Omega(M)\) with
	$
	d_{C^0}(g_i,\psi_i)<\eta.
	$
	Then
	\[
	d_{C^0}(\psi_i,\mathrm{id})
	\le d_{C^0}(\psi_i,g_i)+d_{C^0}(g_i,\mathrm{id})
	< \eta+\eta=2\eta<\delta.
	\]
	Hence, by the local \(C^0\)-generation property, there exists a smooth
	volume-preserving isotopy \(\{\phi_i(s)\}_{s\in[0,1]}\) from
	\(\phi_i(0)=\mathrm{id}\) to \(\phi_i(1)=\psi_i\) such that
	$
	\max_{s\in[0,1]} d_{C^0}(\phi_i(s),\mathrm{id})<\eta.
	$	Define \(K:[0,1]\to\operatorname{Diff}^\infty_\Omega(M)\) inductively by
	\(K(0)=\mathrm{id}\) and, for \(t\in[t_{i-1},t_i]\),
	\[
	K(t):=K(t_{i-1})\circ \phi_i\!\left(\frac{t-t_{i-1}}{t_i-t_{i-1}}\right).
	\]
	Then \(K\) is a smooth volume-preserving isotopy.	We claim that \(d_{C^0}(K(t),h^t)<\varepsilon\) for all \(t\in[0,1]\).
	We prove by induction on \(i\) that
	$
	d_{C^0}(K(t),h^t)<2i\eta
	\qquad\text{for all } t\in[t_{i-1},t_i].
	$
	For \(i=0\) there is nothing to prove. Assume the claim for \(i-1\). For
	\(t\in[t_{i-1},t_i]\), set \(s=(t-t_{i-1})/(t_i-t_{i-1})\). Then,
	\[
	K(t)=K(t_{i-1})\circ \phi_i(s),
	\qquad
	h^t=h^{t_{i-1}}\circ g_{i,s}.
	\]
	Therefore
	\[
	d_{C^0}(K(t),h^t)
	\le d_{C^0}(K(t_{i-1}),h^{t_{i-1}})
	+ d_{C^0}(\phi_i(s),g_{i,s})
	< 2(i-1)\eta + 2\eta = 2i\eta.
	\]
	At \(t=t_i\), we have \(K(t_i)=K(t_{i-1})\circ \psi_i\) and
	\(h^{t_i}=h^{t_{i-1}}\circ g_i\), so
	\[
	d_{C^0}(K(t_i),h^{t_i})
	\le d_{C^0}(K(t_{i-1}),h^{t_{i-1}})
	+ d_{C^0}(\psi_i,g_i)
	< 2(i-1)\eta + \eta < 2i\eta.
	\]
	Thus the induction continues. In particular, for all \(t\in[0,1]\),
	\[
	d_{C^0}(K(t),h^t)<2N\eta = \frac{\varepsilon}{2}<\varepsilon.
	\]
	Since \(H\) and \(\varepsilon\) were arbitrary, smooth volume-preserving
	isotopies are dense in \(\mathcal P\operatorname{Homeo}_0^\Omega(M)\) with
	respect to the \(C^0\)-metric.	Consequently, the time-one maps of smooth volume-preserving isotopies are
	\(C^0\)-dense in \(\operatorname{Homeo}_0^\Omega(M)\). By definition,
	\(\mathbb{G}^\Omega(M)\) is the \(C^0\)-closure of the subgroup generated
	by such time-one maps. Hence, 
	$
	\mathbb{G}^\Omega(M)=\operatorname{Homeo}_0^\Omega(M).
	$
	Therefore the continuous cohomology isomorphism
	\[
	H^1_{\mathrm{cont}}\bigl(\mathbb{G}^\Omega(M),C^0_0(M,\mathbb R)\bigr)
	\cong H^1(M;\mathbb R),
	\]
	from Theorem~\ref{thm:Ismagilov} remains valid with
	\(\mathbb{G}^\Omega(M)\) replaced by
	\(\operatorname{Homeo}_0^\Omega(M)\). This completes the proof.
\end{proof}

	\section*{Declarations}
	\subsection*{Ethics approval}
	Not applicable. 
	\subsection*{Competing interests}
	The author declares that there are no competing interests.
	
	\subsection*{Funding}
	The author received no financial support for the research, authorship, and/or publication of this article.
	\vspace{2em}
	

\end{document}